\documentclass[11pt,letterpaper,twoside,reqno,nosumlimits]{amsart}

\allowdisplaybreaks

\usepackage[usenames,dvipsnames]{xcolor}
\usepackage{fancyhdr}
\usepackage{amsmath,amsfonts,amsbsy,amsgen,amscd,mathrsfs,amssymb,amsthm}
\usepackage{subfig}
\usepackage{url}

\usepackage{mathtools}

\mathtoolsset{showonlyrefs}

\usepackage[font=small,margin=0.25in,labelfont={sc},labelsep={colon}]{caption}

\usepackage{tikz}
\usepackage{microtype}
\usepackage{enumitem}

\definecolor{dark-gray}{gray}{0.3}
\definecolor{dkgray}{rgb}{.4,.4,.4}
\definecolor{dkblue}{rgb}{0,0,.5}
\definecolor{medblue}{rgb}{0,0,.75}
\definecolor{rust}{rgb}{0.5,0.1,0.1}

\usepackage[colorlinks=true]{hyperref}

\hypersetup{urlcolor=rust}
\hypersetup{citecolor=dkblue}
\hypersetup{linkcolor=dkblue}

\usepackage{setspace}

\usepackage{graphicx}
\usepackage{booktabs,longtable,tabu} 
\usepackage{multirow} 

\usepackage{float}

\usepackage[full]{textcomp}

\usepackage[scaled=.98,sups,osf]{XCharter}
\usepackage[scaled=1.04,varqu,varl]{inconsolata}
\usepackage[type1]{cabin}
\usepackage[charter,vvarbb,scaled=1.07]{newtxmath}
\usepackage[cal=boondoxo]{mathalfa}
\usepackage[T1]{fontenc}

\usepackage{bm}

\graphicspath{{figures/}}

\newtheorem{bigthm}{Theorem}

\newtheorem{theorem}{Theorem}[section]
\newtheorem{lemma}[theorem]{Lemma}

\newtheorem{proposition}[theorem]{Proposition}
\newtheorem{fact}[theorem]{Fact}

\newtheorem{corollary}[theorem]{Corollary}

\theoremstyle{definition}

\newtheorem{remark}[theorem]{Remark}

\numberwithin{equation}{section} 
\numberwithin{figure}{section}
\numberwithin{table}{section}

\floatstyle{plaintop}
\newfloat{recipe}{thp}{lor}
\floatname{recipe}{Recipe}
\numberwithin{recipe}{section}

\providecommand{\mathbold}[1]{\bm{#1}}

\renewcommand{\phi}{\varphi}

\newcommand{\eps}{\varepsilon}

\newcommand{\cnst}[1]{\mathrm{#1}} 
\newcommand{\econst}{\mathrm{e}}

\newcommand{\Id}{\mathbf{I}}

\providecommand{\mathbbm}{\mathbb} 

\newcommand{\R}{\mathbbm{R}}

\newcommand{\CC}{\mathbbm{C}}
\newcommand{\N}{\mathbbm{N}}

\newcommand{\M}{\mathbbm{M}}

\newcommand{\abs}[1]{\left\vert {#1} \right\vert}

\newcommand{\Prob}[1]{\mathbbm{P}\left\{{#1}\right\}}

\newcommand{\Expect}{\operatorname{\mathbbm{E}}}

\newcommand{\vct}[1]{\mathbold{#1}}
\newcommand{\mtx}[1]{\mathbold{#1}}

\newcommand{\trace}{\operatorname{tr}}

\newcommand{\psdle}{\preccurlyeq}
\newcommand{\psdge}{\succcurlyeq}

\newcommand{\norm}[1]{\left\Vert {#1} \right\Vert}

\newcommand{\triplenorm}[1]{{\left\vert\kern-0.25ex\left\vert\kern-0.25ex\left\vert #1
    \right\vert\kern-0.25ex\right\vert\kern-0.25ex\right\vert}}

\newcommand{\mF}{\mathcal{F}}

\newcommand{\E}{\Expect}

\begin{document}

\title{Matrix Concentration for Products}
\author[Huang et al.]{De Huang, Jonathan Niles-Weed, Joel A.~Tropp, and Rachel Ward}
\thanks{The authors gratefully acknowledge the funding for this work.
DH was supported under NSF grant DMS-1613861.
JNW and RW were supported in part by the Institute for Advanced Study,
where some of this research was conducted.
JAT was supported under ONR Awards N00014-17-1-2146 and N00014-18-1-2363.
RW also received support from AFOSR MURI Award N00014-17-S-F006.}
\date{4 March 2020}

\begin{abstract}
This paper develops nonasymptotic growth and concentration bounds for a product of independent random matrices.
These results sharpen and generalize recent work of Henriksen--Ward, and they are similar in spirit to the results of Ahlswede--Winter and of Tropp
for a sum of independent random matrices.
The argument relies on the uniform smoothness properties of the Schatten trace classes.
\end{abstract}
\maketitle

\section{Motivation}

Products of random matrices arise in many contemporary applications in the mathematics
of data science.  For instance, they describe the evolution of stochastic linear
dynamical systems, which include popular stochastic algorithms for optimization such as Oja's algorithm for streaming principal component analysis \cite{Oja82:Simplified-Neuron} and the randomized Kaczmarz method for solving linear systems~\cite{SV09:Randomized-Kaczmarz}.
To understand the detailed behavior of these algorithms, such as the rate of
convergence, we may seek out methods for studying a product of random matrices.

Unfortunately, the tools currently available in the literature
are poorly adapted to these circumstances.
Indeed, an instantiation of a stochastic optimization algorithm
involves a finite product of finite-dimensional matrices,
often with a particular structure (e.g., low-rank perturbations of the identity).
But most existing theoretical results are limit laws that require the number
of factors in the product or the dimension of the factors to
tend to infinity.  Furthermore, strong assumptions
on the random matrices (e.g., independent and identically distributed entries) are usually required.

This paper offers some new tools for studying random matrix products
that arise from stochastic optimization algorithms and related problems.
The research is inspired by the recent paper~\cite{henriksen2019concentration} of Henriksen and Ward.
Our hope is to replicate the successful program for studying sums of random matrices,
implemented in the works~\cite{AW02:Strong-Converse,Oli09:Concentration-Adjacency,tropp2011freedman,Tro12,Tro15:Introduction-Matrix,Tro16:Expected-Norm}.
In particular, we seek to develop methods that are flexible, easy to use, and powerful~\cite{Tro18:Second-Order-Matrix}.
We also aspire to use transparent theoretical arguments that can be adapted easily to new situations.

\section{Contributions}

To motivate our work, we start with an elementary concentration
inequality for a product of independent random numbers.
We will generalize this bound, and others, to the matrix setting.

\subsection{Context: A Product of Random Numbers Near $1$}
\label{sec:product-scalars}

Consider an independent family $\{ X_1, X_2, \dots \} \subset \R$
of bounded random variables that satisfy
$$
\E X_i = \mu
\quad\text{and}\quad
\abs{ X_i - \mu }^2 \leq b^2 \quad\text{almost surely.}
$$
Form a product of random perturbations of $1$, and compute its mean:
$$
Z_n := \prod_{i=1}^n \left( 1 + \frac{X_i}{n} \right)
	\quad\text{and}\quad
	\E Z_n = \left(1 + \frac{\mu}{n}\right)^n = \econst^{\mu} \cdot ( 1 - O(n^{-1}) ).
$$
We anticipate that the random product $Z_n$ concentrates
around its expectation $\E Z_n \approx \econst^{\mu}$.

To check this surmise, we can use standard methods from scalar concentration theory.
For $s > 0$,
$$
\begin{aligned}
\Prob{ Z_n \geq (1 + s) \, \econst^{\mu} }
	&= \Prob{ \prod_{i=1}^n \left(1 + \frac{X_i}{n}\right) \geq (1 + s) \,\econst^{\mu} } \\
	&\leq \Prob{ \exp\left( \frac{1}{n} \sum_{i=1}^n X_i \right) \geq (1 + s)\, \econst^{\mu} } \\
	&= \Prob{ \frac{1}{n} \sum_{i=1}^n (X_i - \E X_i) \geq \log(1 + s) }. 
\end{aligned}
$$
The inequality follows from the numerical fact $1 + a \leq \econst^a$, valid for $a \in \R$.
Hoeffding's inequality furnishes the bound
\begin{equation} \label{eq:growth_inf}
\Prob{ Z_n \geq (1 + s)\, \econst^{\mu} }
	\leq \exp\left( \frac{-n \log^2(1+s)}{2b^2} \right).
\end{equation}
At the small scale $s \leq \econst$, in which case $\log(1 + s) \geq s / \econst$,
the growth bound~\eqref{eq:growth_inf} implies a subgaussian tail behavior:
\begin{equation}\label{eq:concentration_inf}
\Prob{ Z_n - \E Z_n \geq t\, \econst^{\mu}} \leq \Prob{ Z_n - \econst^{\mu} \geq t\,\econst^{\mu} }
	\leq \exp\left( \frac{-nt^2}{2\econst^2 b^2} \right)
	\quad\text{for $t \leq \econst$.}
\end{equation}
A similar inequality holds for the lower tail.

\subsection{A Product of Random Perturbations of the Identity}
\label{sec:product-matrices-intro}

We might hope that products of random matrices exhibit a similar behavior.
Consider an independent family $\{ \mtx{X}_1, \dots, \mtx{X}_n \} \subset \M_d$
of $d \times d$ matrices that satisfy
\begin{equation} \label{eqn:mtx-hypothesis-intro}
\Expect \mtx{X}_i = \mtx{A}
\quad\text{and}\quad
\norm{ \mtx{X}_i - \E \mtx{X}_i }^2 \leq b^2
\quad\text{almost surely}.
\end{equation}
Here are elsewhere, $\norm{ \cdot }$ is the spectral norm, that is, the $\ell_2$ operator norm.
Form a product of random perturbations of the identity and compute its mean:
\begin{equation} \label{eqn:mtx-product-intro}
\mtx{Z}_n = \left( \Id + \frac{\mtx{X}_n}{n} \right) \cdots \left( \Id + \frac{\mtx{X}_1}{n} \right)
\quad\text{and}\quad
\E \mtx{Z}_n = \left( \Id + \frac{\mtx{A}}{n} \right)^n \approx \econst^{\mtx{A}}.
\end{equation}
Is it true that the spectral norm $\norm{ \mtx{Z}_n }$ is proportional to $\econst^\mu$,
where $\mu = \norm{\mtx{A}}$?
Does the random product $\mtx{Z}_n$ concentrate near its mean $\E \mtx{Z}_n$?

These speculations are correct.  Moreover, we can obtain bounds that parallel
the scalar inequalities announced in the last subsection.  Here is one particular
result that follows from our analysis.

\begin{bigthm}[Products of Perturbations of the Identity---Special case] 
\label{thm:products-intro}
Consider an independent family $\{ \mtx{X}_1, \dots, \mtx{X}_n \} \subset \M_d$ of random matrices
that satisfy the hypotheses~\eqref{eqn:mtx-hypothesis-intro}.  Define $\mu := \norm{\mtx{A}}$.
The matrix product $\mtx{Z}_n$ introduced in~\eqref{eqn:mtx-product-intro} satisfies the bounds
\begin{align}
\Prob{ \norm{\mtx{Z}_n} \geq (1+s) \, \econst^{\mu} }
	&\leq d \cdot \exp\left( \frac{-n \log^2 (1+s)}{2b^2} \right)
	&&\text{when $\log(1+s) \geq 2b^2/n$};
	\label{eqt:int1} \\
\Prob{ \norm{\mtx{Z}_n - \E \mtx{Z}_n} \geq t \econst^{\mu} }
	&\leq (d + \econst) \cdot \exp\left(\frac{-n t^2}{2\econst^2 b^2}\right)
	&&\text{when $t \leq \econst$.}
	\label{eqt:int2}
\end{align}
\end{bigthm}

\noindent
Theorem~\ref{thm:products-intro} follows from Corollary~\ref{cor:expectation_Oja}.

As compared with the scalar bounds~\eqref{eq:growth_inf} and~\eqref{eq:concentration_inf},
the results in Theorem~\ref{thm:products-intro} feature an additional dimensional factor $d$
in front of the exponential.  This term leads to a dependency of $\log d$ in the bounds for
products of random matrices.  Otherwise, everything is the same, including the constants.

\subsection{Proof Strategy}

How might one establish a result like Theorem~\ref{thm:products-intro}?
The derivation in Section~\ref{sec:product-scalars} is valid only for products of random scalars.
We cannot even begin to make this argument for matrices because the exponential
of a sum of matrices generally does not equal the product of the exponentials.

In this paper, we take a completely different approach.  The key is to observe that
multiplying a random product $\mtx{Z} \in \M_d$ by a statistically independent factor $\mtx{Y} \in \M_d$
creates a predictable change plus a random perturbation:
$$
\mtx{Y} \mtx{Z}
	= (\E \mtx{Y}) \mtx{Z} + (\mtx{Y} - \E \mtx{Y}) \mtx{Z}.
$$
Since the second term has zero mean, conditional on $\mtx{Z}$,
we can exploit this orthogonality property to estimate the size
of the product:
$$
\begin{aligned}
\E \norm{ \mtx{Y} \mtx{Z} }_2^2
	&= \E \norm{ (\E \mtx{Y}) \mtx{Z} }_2^2 + \E \norm{ (\mtx{Y} - \E \mtx{Y}) \mtx{Z} }_2^2 \\
	&\leq \big( \norm{ \E \mtx{Y} }^2 + \E \norm{ \mtx{Y} - \E \mtx{Y} }^2 \big)\big( \E \norm{\mtx{Z}}_2^2 \big)
	=:  (1 + v)\, m \cdot \big( \E \norm{\mtx{Z}}_2^2 \big)
\end{aligned}
$$
The notation $\norm{\cdot}_2$ refers to the Schatten $2$-norm, also known as the Frobenius norm.
The last step introduces data about the random matrix $\mtx{Y}$: the mean $m = \norm{ \E \mtx{Y} }$
and the relative variance $v = \E \norm{\mtx{Y} - \E \mtx{Y}}^2 / \norm{ \E \mtx{Y} }^2$.
We can apply the same argument recursively to decompose the matrix $\mtx{Z}$
into its own factors.

The approach in the last paragraph depends on the fact
that $\norm{\cdot}_2$ is the norm induced by the trace inner product.
To undertake the same action for the spectral norm $\norm{\cdot}$,
we first need to approximate the spectral norm by the Schatten $p$-norm
for $p \approx \log d$.  Then we can invoke a remarkable geometric property
of the Schatten $p$-norm, called \emph{uniform smoothness}, 
as a substitute for the orthogonality law.
See the paper~\cite{naor2012banach} for an
introduction to this circle of ideas.
Section~\ref{sec:unif-smoothness} executes this method.

\subsection{Additional Results}

We establish a family of norm inequalities for products of random matrices.
The main result, Theorem~\ref{thm:pnorm}, gives a bound for the moments of a
Schatten $p$-norm of a random product and a centered random product.  From this fact,
we derive expectation bounds, tail bounds, and matrix concentration inequalities.
Many of these results hold under weaker assumptions than Theorem~\ref{thm:products-intro},
addressing cases where the matrices have different means or are unbounded.

To give a better indication of what we can prove, let us give an informal
presentation of one of our main results, Corollary~\ref{cor:expectation}.
The statement concerns a general product $\mtx{Z}_n = \mtx{Y}_n \cdots \mtx{Y}_1$ of independent
random matrices of dimension $d$.  Abbreviating $p = 1 + 2 \log d$, we have the inequality 
$$
\E \norm{ \mtx{Z}_n - \E \mtx{Z}_n } \leq \econst \sqrt{pv} \prod_{i=1}^n \norm{ \E \mtx{Y}_i }
\quad\text{when}\quad
v := \sum_{i=1}^n \frac{\E \norm{ \mtx{Y}_i - \E\mtx{Y}_i }^2 }{ \norm{ \E \mtx{Y}_i }^2 } \leq \frac{1}{p}.
$$
We can interpret $v$ as the accumulated relative variance in the product.

For example, in the setting of Theorem~\ref{thm:products-intro}, the quantity $v = O(b^2 / n)$.
It follows that
\begin{equation} \label{eqn:E-product-intro}
\E \norm{ \mtx{Z}_n - \E \mtx{Z}_n } = O\left( \sqrt{\frac{pb^2}{n}} \norm{ \E \mtx{Z}_n } \right)\,.
\end{equation}
In particular, $\norm{ \mtx{Z}_n }$ is much closer to $\econst^{\mu}$ than
to the worst-case bound $\econst^{b}$.

\subsection{Roadmap}

We continue with an overview of related work in Section~\ref{sec:related}.
Section~\ref{sec:unif-smoothness} presents background results from
matrix theory and high-dimensional probability.  We establish our
main results for general matrix products in Section~\ref{sec:main-results}.
Afterward, Section~\ref{sec:perturbations} draws corollaries for
a product of perturbations of the identity.  Finally, we describe
some refinements and extensions in Section~\ref{sec:extensions}.

\section{Related Work}
\label{sec:related}

Products of random matrices have been studied for decades,
primarily within the fields of ergodic theory, control theory,
random matrix theory, and free probability.  More recently,
applied mathematicians have developed results that are tailored
to problems arising in data science.
Almost all prior work is either asymptotic in the length of the
product or asymptotic in the dimension of the matrices.
This section contains an overview of these inquiries.

\subsection{Direct Connections}

The most immediate precedent for our research is the recent paper of
Henriksen and Ward~\cite{henriksen2019concentration}.  They were
motivated by the problem of understanding streaming algorithms
for covariance estimation.
Their work gives, perhaps, the first explicit nonasymptotic bounds for a
somewhat general product of random matrices with fixed dimension.  The argument
is based on the matrix Bernstein inequality and a combinatorial fact about set partitions.

Henriksen and Ward focus on the setting of Theorem~\ref{thm:products-intro},
and they establish a bound of the form
$$
\Expect \norm{ \mtx{Z}_n - \E \mtx{Z}_n } \leq \frac{b \econst^b}{\sqrt{n}} \cdot \mathrm{polylog}(n,d).
$$
In contrast, our new result~\eqref{eqn:E-product-intro} replaces the
worst-case factor $\econst^b$ with the more typical value $\econst^{\mu}$.
We are also able to relax several of the assumptions in~\cite{henriksen2019concentration}.

Also in the setting of Theorem~\ref{thm:products-intro}, several works obtain results on the asymptotic behavior of $\mtx{Z}_n$.
Berger~\cite{Ber84} establishes, via a semigroup argument based on the Chernoff product formula, that $\mtx{Z}_n \to \econst^{\mtx A}$ in probability as $n \to \infty$.
Emme and Hubert~\cite{emme2017limit} recently obtained a refinement of this result: motivated by a problem in ergodic theory, they show that $\mtx{Z}_n \to \econst^{\mtx{A}}$ as $n \to \infty$ under the sole assumptions that $ \sum_{i=1}^n \mtx{X}_i/n \to \mtx A$ and $\sum_{i=1}^n \|\mtx{X}_i\|/n < \infty$.
Their argument expands the product and computes the limit of the
$k$th order term using an induction.
Neither approach readily yields nonasymptotic bounds.

\subsection{Other Recent Applications}

Some applied work on random matrix products has been driven by the empirical
observation that stochastic gradient descent converges faster when
the gradient approximations are sampled \emph{without} replacement, rather
than sampled \emph{with} replacement.  Some papers that investigate this
question from the point of view of (nonasymptotic) matrix inequalities
include~\cite{RR12:Beneath-Valley,IKW16:Arithmetic-Geometric,AJZ17:Noncommutative-Versions}.
This specific problem has been solved by G{\"u}rb{\"uz}balaban et al.~\cite{GOP19:Random-Reshuffling}
using optimization theory. However, none of these results directly address the questions at hand.

Researchers studying randomly initialized deep neural networks have also
developed theoretical analysis for products of random matrices;
see~\cite{hanin2018products,Yan19:Scaling-Limits}.
These results involve operations on matrices with independent entries,
and they focus on the large-matrix limit.

\subsection{Ergodic Theory and Control Theory}

Products of random matrices describe the evolution of a linear stochastic
dynamical system.  For this reason, they have been a subject of perennial
interest within the literatures on ergodic theory and on control theory.  For
the most part, this research is concerned with properties of the asymptotics
of infinite products of matrices (of fixed size).  Let us give a few more details.

Consider a finite family $\mathcal{A} = \{ \mtx{A}_1, \dots, \mtx{A}_s \} \subset \M_d$
of fixed matrices.  Construct a random matrix $\mtx{X} \in \M_d$ with
the distribution
$$
\Prob{ \mtx{X} = \mtx{A}_i } = \frac{1}{s}
\quad\text{for each $i = 1, \dots, s$.}
$$
The \emph{Lyapunov exponent} of the set $\mathcal{A}$ is the quantity
$$
\lambda(\mathcal{A}) := \lim_{n \to \infty} \frac{1}{n} \log \norm{ \mtx{X}_n \cdots \mtx{X}_1 }
	\quad\text{where $\mtx{X}_i \sim \mtx{X}$ iid.}
$$
The Furstenberg--Kesten theorem~\cite{FK60:Products-Random} establishes
that $\lambda(\mathcal{A})$ exists almost surely,
but approximating $\lambda(\mathcal{A})$ is algorithmically
undecidable~\cite[Thm.~2]{TB97:Lyapunov-Exponent}.  As a consequence,
we must be pessimistic about finding a completely satisfactory solution
to the matrix concentration problem for products.

To learn more about Lyapunov exponents and to find additional references,
see the paper~\cite{AP19:Lyapunov-Exponent} for work in control theory
and the paper~\cite{Wil17:Lyapunov-Exponents} for work in ergodic theory.
Another major application of random products is to study the asymptotic
behavior of a random walk on a group;
we refer the reader to~\cite{Led01:Some-Asymptotic,Fur02:Random-Walks,BQ16:Random-Walks}
for more information.

\subsection{Random Matrix Theory and Free Probability}

Products of random matrices have also been considered
within random matrix theory and free probability.
This connection is natural, but 
matrix products have received somewhat less attention
than other kinds of random matrix models.  In these
contexts, it is common to study a product of a small
number of matrices (two or three, say) in the limit
as the dimension of the matrices grows.

Bai and Silverstein~\cite[Chap.~4]{BS10:Spectral-Analysis}
present a limit law for the sequence of products
of a random matrix with iid entries and a random matrix
whose spectral distribution has a deterministic limit.
This theorem is motivated by a statistical application,
multivariate analysis of variance.  Note, however,
that convergence of the spectral distribution does
not determine the limit of the spectral norm.

Free probability gives a complete description of the
spectral distribution of a product of two freely
independent elements as the ``multiplicative free convolution''
of the spectral distributions of the factors.
The connection to random matrix theory stems
from the fact that a family of ``adequately random'' matrices becomes
freely independent in the limit as the dimension of the
matrices tends to infinity.  See the book of
Nica \& Speicher~\cite{NS06:Lectures-Combinatorics}
for a digestible introduction; some other good treatments
include~\cite{RE08:Polynomial-Method,Shl19:Random-Matrices,Spe19:Lecture-Notes}.
Free probability has significant applications in wireless communications~\cite{TV04:Random-Matrix}.

For highly structured random matrices (invariant ensembles), it may be possible
to obtain more detailed formulas for products.
See~\cite{kieburg2019products,dartois2019schwinger}
for some recent work in this direction.

\section{Random Matrix Inequalities via Uniform Smoothness}
\label{sec:unif-smoothness}

To analyze products of random matrices, we exploit classic
methods that were developed to study the evolution of a
martingale taking values in a uniformly smooth Banach space.
These ideas are relevant for us because the matrix Schatten
classes (with power $2 \leq p < \infty$) enjoy a remarkable
uniform smoothness property.

In this section, we outline the required background
from matrix analysis and high-dimensional probability.
Naor's tutorial paper~\cite{naor2012banach} serves
as a model for our presentation, and it contains a
more general treatment.
See Section~\ref{sec:unif-smooth-history} for additional
discussion about the history of these ideas.

\subsection{Notation and Background}

We work in the complex field $\CC$; identical results hold
for the real field $\R$.  We often use the infix notation for
the minimum ($\wedge$) and the maximum ($\vee$) of two real numbers.

The operator $\mathbbm{P}$ computes the probability on an event.
The operator $\E$ computes the expectation of a random variable.
Subscripts denote partial expectation; for example, $\E_{Z}$ is
the expectation over the randomness in $Z$.  Nonlinear
functions, such as powers, bind before the expectation.

The linear space $\CC^{d \times r}$ contains all $d \times r$
matrices with complex entries.
The algebra $\M_d$ consists of all $d \times d$ matrices with
complex entries.  We use the standard definitions of scalar
multiplication, matrix addition, matrix multiplication,
and the adjoint (i.e., conjugate transpose).
Any statement about matrices that is not qualified with
specific dimensions holds for all matrices with compatible dimensions.
Nonlinear functions, such as matrix powers, bind before the trace.
The matrix absolute value $\abs{ \mtx{A} } := (\mtx{A}^* \mtx{A})^{1/2}$,
where $(\cdot)^{1/2}$ is the positive-semidefinite square root
of a positive-semidefinite matrix.

We write $\norm{\cdot}$ for the spectral norm on matrices;
the spectral norm coincides with the maximum singular value,
and it is also known as the $\ell_2$ operator norm.
For each $p \geq 1$, the symbol $\norm{\cdot}_p$ refers to the Schatten $p$-norm
which returns the $\ell_p$ norm of the singular values of its argument.
The symbol $S_p$ refers to a linear space of matrices (of fixed dimension),
equipped with the Schatten $p$-norm.

For parameters $p, q \geq 1$, we define the $L_q(S_p)$ norm of a random matrix $\mtx{X}$ 
as 
$$
\triplenorm{ \mtx{X} }_{p,q} := \norm{ \mtx{X} }_{L_q(S_p)} := \big( \E \norm{\mtx{X}}_p^q \big)^{1/q}.
$$
The $L_q(S_p)$ norm is an operator ideal norm, in the sense that
\begin{equation} \label{eqn:operator-ideal}
\triplenorm{ \mtx{AX} }_{p,q} \leq \norm{\mtx{A}} \cdot \triplenorm{\mtx{X}}_{p,q}
	\quad\text{for fixed $\mtx{A}$ and random $\mtx{X}$.}
\end{equation}
This statement follows instantly from the analogous property of the Schatten $p$-norm.

We sometimes use the following simple inequalities for the moments
of a random matrix $\mtx{X}$:
\begin{equation} \label{eqn:moment-method}
\E\|\mtx{X}\| \leq \inf_{p\geq 1} \ \E\|\mtx{X}\|_p = \inf_{p, q \geq 1} \ \triplenorm{ \mtx{X} }_{p,q}.
\end{equation}
The equality follows from Lyapunov's inequality, combined with the fact that $\triplenorm{ \mtx{X} }_{p, 1} = \E \|\mtx X\|_p$ for all $p \geq 1$.

\subsection{Uniform Smoothness for Matrices}

Uniform smoothness
\footnote{More precisely, we are considering uniformly smooth
spaces whose modulus of smoothness has power type $2$.}
is a property of a normed space that describes
how much the norm of a point changes under symmetric perturbation.
Since the Schatten-2 space $S_2$ is an inner-product space,
the parallelogram law gives an exact description of this
phenomenon:
$$
\frac{1}{2} \left[ \norm{\mtx{X} + \mtx{Y}}_2^2 + \norm{\mtx{X} - \mtx{Y}}_2^2 \right]
	= \norm{\mtx{X}}_2^2 + \norm{\mtx{Y}}_2^2.
$$
Remarkably, in other Schatten classes, the parallelogram law
is replaced by an inequality.

\begin{fact}[Uniform Smoothness for Schatten Classes] \label{fact:bcl}
Let $\mtx{A}, \mtx{B}$ be matrices of the same size.  For $p \geq 2$,
\begin{equation} \label{eqn:2smooth}
\left[ \frac{1}{2} \left(\norm{ \mtx{A} + \mtx{B} }_p^p + \norm{\mtx{A} - \mtx{B} }_p^p \right) \right]^{2/p}
	\leq \norm{ \mtx{A} }_{p}^2 + \cnst{C}_p \norm{\mtx{B}}_p^2.
\end{equation}
The optimal constant $\cnst{C}_p := p - 1$.  The inequality is reversed
when $1 \leq p \leq 2$.
\end{fact}

Fact~\ref{fact:bcl} was first established by Tomczak-Jaegermann~\cite{TJ74:Moduli-Smoothness};
she obtained the sharp constant $\cnst{C}_p$ when $p$ is an even number.
Ball, Carlen, and Lieb~\cite[Thm.~1]{ball1994sharp} determined that $\cnst{C}_p$
is the optimal constant for all values of $p$.  Throughout the paper,
we will continue to write $\cnst{C}_p = p - 1$.

\subsection{Uniform Smoothness for Random Matrices}

Much as the Schatten class $S_p$ of matrices enjoys a uniform smoothness property,
the normed space $L_q(S_p)$ of random matrices is also uniformly smooth.  When
$2 \leq q \leq p$, this statement follows as an easy consequence of Fact~\ref{fact:bcl}.

\begin{corollary}[Uniform Smoothness for Random Matrices] \label{cor:2smooth}
Let $\mtx{X}, \mtx{Y}$ be random matrices of the same size.
When $2 \leq q \leq p$,
\[
\left[ \frac{1}{2} \left( \triplenorm{ \mtx{X} + \mtx{Y} }_{p,q}^q + \triplenorm{\mtx{X} - \mtx{Y}}_{p,q}^q \right) \right]^{2/q}
	\leq \triplenorm{ \mtx{X} }_{p,q}^2 + \cnst{C}_p \triplenorm{\mtx{Y}}_{p,q}^2.
\]
\end{corollary}

\begin{proof}
Apply Lyapunov's inequality to the left-hand side of~\eqref{eqn:2smooth}
to pass from the $p$th power to the $q$th power, and then transfer the
exponent to the right-hand side to obtain the pointwise bound
\[
\frac{1}{2} \left( \norm{ \mtx{X} + \mtx{Y} }_p^q + \norm{\mtx{X} - \mtx{Y} }_p^q \right)
	\leq \left[ \norm{ \mtx{X} }_{p}^2 + \cnst{C}_p \norm{\mtx{Y}}_p^2 \right]^{q/2}.
\]
Take the expectation, and use the triangle inequality for the $L_{q/2}$ norm:
\[
\frac{1}{2} \left( \E \norm{ \mtx{X} + \mtx{Y} }_p^q + \E \norm{\mtx{X}- \mtx{Y}}_p^q  \right)
	\leq \left[  \big( \E \norm{\mtx{X}}_{p}^q \big)^{2/q} + \cnst{C}_p \big( \E \norm{\mtx{Y}}_{p}^q \big)^{2/q} \right]^{q/2}.
\]
Reinterpret the latter display using the $L_q(S_p)$ norm $\triplenorm{\cdot}_{p,q}$.
\end{proof}

\subsection{Subquadratic Averages for Random Matrices}

Corollary~\ref{cor:2smooth} admits a powerful extension that
controls how the norm of a matrix changes if we add a random
matrix that has zero mean.  This result is the main tool that
we employ in our study of random products.

\begin{proposition}[Subquadratic Averages] 
\label{prop:Ricard_Xu}
Consider random matrices $\mtx{X}, \mtx{Y}$ of the same size that satisfy
$\E[ \mtx{Y} | \mtx{X} ] = \mtx{0}$.  When $2 \leq q \leq p$,
\[
\triplenorm{ \mtx{X} + \mtx{Y} }_{p,q}^2
	\leq \triplenorm{\mtx{X}}_{p,q}^2 + \cnst{C}_p \triplenorm{\mtx{Y}}_{p,q}^2.
\]
The constant $\cnst{C}_p = p - 1$ is the best possible.
\end{proposition}

Ricard and Xu~\cite{ricard2016noncommutative} obtained a version
of Proposition~\ref{prop:Ricard_Xu} in the more general setting
of a von Neumann algebra.  In their work, the expectation implicit in
the $L_q$ norm is replaced by the projection onto a subalgebra.
They emphasize that the key feature of their work is the determination
of the sharp constant.

Here, we offer a very short proof of Proposition~\ref{prop:Ricard_Xu}
with a suboptimal constant.  The method is drawn from Naor's paper~\cite{naor2012banach}.
Lemma~\ref{lem:subquadratic-sharp}, in the appendix, unspools an elementary argument
that delivers the sharp constant.

\begin{proof}
By Jensen's inequality, applied conditionally on $\mtx{X}$,
\begin{equation*}
\begin{aligned}
\frac{1}{2} \left( \triplenorm{ \mtx{X} + \mtx{Y} }_{p,q}^2 + \triplenorm{ \mtx{X} }_{p,q}^2 \right)
	&\leq \frac{1}{2} \left( \triplenorm{ \mtx{X} + \mtx{Y} }_{p,q}^2 + \triplenorm{ \mtx{X} - \mtx{Y} }_{p,q}^2\right)  \\
	&\leq \left[ \frac{1}{2} \left( \triplenorm{ \mtx{X} + \mtx{Y} }_{p,q}^q + \triplenorm{ \mtx{X} - \mtx{Y} }_{p,q}^q \right) \right]^{2/q}
	\leq \triplenorm{ \mtx{X} }_{p,q}^2 + \cnst{C}_p \triplenorm{ \mtx{Y} }_{p,q}^2.
\end{aligned}
\end{equation*}
The second inequality is Lyapunov's; the third is Corollary~\ref{cor:2smooth}.
Upon rearranging, we find that
\begin{equation} \label{eqn:subquadratic-weak}
\triplenorm{ \mtx{X} + \mtx{Y} }_{p,q}^2 \leq \triplenorm{\mtx{X}}_{p,q}^2 + 2 \cnst{C}_p \triplenorm{\mtx{Y}}_{p,q}^2.
\end{equation}
This is the stated result, with a spurious factor of $2$. 
\end{proof}

\subsection{Matrix-Valued Martingales}

To demonstrate the value of Proposition~\ref{prop:Ricard_Xu}, let us
explain how it leads to moment bounds for a matrix-valued martingale sequence.
Consider a null matrix martingale $\{ \mtx{X}_1, \dots, \mtx{X}_n \} \subset \M_d$
with difference sequence $\{ \mtx{\Delta}_1, \dots, \mtx{\Delta}_n \} \subset \M_d$.
That is,
\[
\mtx{X}_0 = \mtx{0}
\quad\text{and}\quad
\mtx{X}_{i} = \mtx{X}_{i-1} + \mtx{\Delta}_{i}
\quad\text{where}\quad
\Expect[ \mtx{\Delta}_i | \mtx{X}_0, \dots, \mtx{X}_{i-1} ] = \mtx{0}
\quad\text{for $i = 1, \dots, n$.}
\]
Applying Proposition~\ref{prop:Ricard_Xu} repeatedly, we arrive at the bound
\begin{equation} \label{eqn:mtx-martingale}
\triplenorm{ \mtx{X}_n }_{p,q}^2
	\leq \cnst{C}_p \sum_{i=1}^n \triplenorm{ \mtx{\Delta}_i }_{p,q}^2.
\end{equation}
In words, the squared norm of the martingale is controlled
by the sum of the squares of the norms of the martingale differences.
The inequality~\eqref{eqn:mtx-martingale} is a powerful extension of the orthogonality
of the increments of a martingale taking values in an inner-product space, say $S_2$.
The uniform smoothness constant $\cnst{C}_p$ shows how the geometry of the matrix space
intermediates.

In this work, we will develop bounds for random matrix products
by applying a similar technique to appropriately chosen
decompositions of the product.

\subsection{History}
\label{sec:unif-smooth-history}

The approach in this section has a long history.  Let us summarize the contributions
that are most relevant to our development.

For real numbers, the (sharp) uniform
smoothness property in Fact~\ref{fact:bcl} is known as the \emph{two-point inequality};
it was established independently by Leonard Gross~\cite{Gro72} and Aline Bonami~\cite{Bon70} in the early 1970s,
with later contributions by William Beckner~\cite{Bec75}.  In 1974, the uniform smoothness property
for the Schatten classes was obtained by Nicole Tomczak-Jaegermann~\cite{TJ74:Moduli-Smoothness}.
It took another 20 years before Ball, Carlen, and Lieb~\cite{ball1994sharp}
obtained the sharp uniform smoothness constants for all Schatten classes.  The property dual to uniform
smoothness is called \emph{uniform convexity}.  See~\cite{ball1994sharp} for a detailed exposition.

Tomczak-Jaegermann~\cite[Thm.~3.1]{TJ74:Moduli-Smoothness} also demonstrated that Rademacher averages
are subquadratic in each Schatten space $S_p$ with $p \geq 2$; that is, the Banach space $S_p$ is \emph{type 2}~\cite{LedTal11}.
This fact is a prototype
for the more general result stated in Proposition~\ref{prop:Ricard_Xu}.
Tropp~\cite[Sec.~4.8]{Tro12} points out that parts of the
Ahlswede--Winter~\cite[App.]{AW02:Strong-Converse} theory of sums of independent random matrices
already follow from Tomczak-Jaegermann's work.  (In contrast, Tropp's matrix concentration inequalities~\cite{Tro12}
are more closely related to a fact from operator theory, the noncommutative Khintchine inequality
of Fran\c{c}oise Lust-Piquard~\cite{LP86:Inegalites-Khintchine}; Tropp's results are derived using
a theorem~\cite[Thm.~6]{Lie73:Convex-Trace} of Elliot Lieb.)

Assaf Naor~\cite{naor2012banach} traces the application of uniform convexity inequalities
in the study of martingales to a 1975 paper of Gilles Pisier~\cite{Pis75:Martingales-Values}.
Naor~\cite{naor2012banach} gives a nice introduction to this circle of ideas, which
he uses to derive a general version of the Azuma inequality that holds in any uniformly smooth Banach space.

At least as early as 1988, Donald Burkholder~\cite{Bur88:Sharp-Inequalities} applied
closely related convexity inequalities to derive sharp inequalities for martingales
taking values in a Hilbert space.
The paper~\cite{ricard2016noncommutative} of {\'E}ric Ricard and Quanhua Xu is a
recent entry in this line of research.

\section{A Product of Independent Random Matrices}
\label{sec:main-results}

In this section, we obtain our main results on the
growth and concentration of a product of independent
random matrices.
Section~\ref{sec:decomp} shows how to decompose a random product
into pieces that we can control using a recursive argument.
Based on these ideas, we derive Theorem~\ref{thm:pnorm},
a general bound on the moments of the norm of the matrix product.
The moment estimate leads to a family of expectation bounds (Corollary~\ref{cor:expectation})
and probability bounds (Corollary~\ref{cor:concentration}).

The balance of the paper contains applications of these results
(Section~\ref{sec:perturbations}) and extensions of the method
to other settings (Section~\ref{sec:extensions}).

\subsection{Decomposition of Random Products}
\label{sec:decomp}

Our approach is based on a recursive argument that describes
how the product evolves as we include more factors.  At each
step, we decompose the product into a nonrandom term and a
random term with mean zero.  This formulation allows us
to apply Proposition~\ref{prop:Ricard_Xu} on subquadratic averages.

Consider a fixed matrix $\mtx{Z}_0 \in \M_d$ and an independent family $\{ \mtx{Y}_1,  \mtx{Y}_2, \dots, \mtx{Y}_n  \} \subset \M_d$
of random matrices.  We can recursively construct products of these random matrices:
\[
\mtx{Z}_{i} = \mtx{Y}_i \mtx{Z}_{i-1}
\quad\text{for $i = 1, \dots, n$.}
\]
Evidently, the last element of the sequence takes the form $\mtx{Z}_n = \mtx{Y}_n \cdots \mtx{Y}_1 \mtx{Z}_0$.
By independence, $\E \mtx{Z}_n = (\E \mtx{Y}_n) \cdots (\E \mtx{Y}_1) \mtx{Z}_0$.

The random product $\mtx{Z}_i$ admits a simple decomposition into a mean term and a fluctuation term:
\begin{equation} \label{eqn:prod-decomp}
\mtx{Z}_i = \mtx{Y}_i\mtx{Z}_{i-1} = (\E \mtx{Y}_i)\mtx{Z}_{i-1} + (\mtx{Y}_i-\E \mtx{Y}_i)\mtx{Z}_{i-1}
	\quad\text{for each $i = 1, \dots, n$.}
\end{equation}
Since $\mtx{Y}_i$ is independent from $\mtx{Z}_{i-1}$, the second term is conditionally zero mean:
\begin{equation} \label{eqn:fluctuation-mean}
\E[ (\mtx{Y}_i - \E \mtx{Y}_i) \mtx{Z}_{i-1} | \mtx{Z}_{i-1} ] = \mtx{0}.
\end{equation}
The property~\eqref{eqn:fluctuation-mean} supports the use of Proposition~\ref{prop:Ricard_Xu}.
It is also helpful to have an explicit norm bound for the random fluctuation term:
\begin{equation} \label{eqn:fluctuation-bd}
\begin{aligned}
\triplenorm{ (\mtx{Y}_i - \E \mtx{Y}_i) \mtx{Z}_{i-1} }_{p,q}
	\leq \left( \E \norm{ \mtx{Y}_i - \E \mtx{Y}_i }^q \cdot \Expect \norm{ \mtx{Z}_{i-1} }_p^q \right)^{1/q}
	= \left( \E \norm{ \mtx{Y}_i - \E \mtx{Y}_i }^q \right)^{1/q} \triplenorm{ \mtx{Z}_i }_{p,q}. 
\end{aligned}	
\end{equation}
The first relation follows from the operator ideal property of the Schatten $p$-norm
and the statistical independence of the random matrices $\mtx{Y}_i$ and $\mtx{Z}_{i-1}$.

We can study the concentration properties of the product $\mtx{Z}_i$
using a related decomposition:
\begin{equation} \label{eqn:conc-decomp}
\mtx{Z}_i - \E \mtx{Z}_i = \mtx{Y}_i\mtx{Z}_{i-1} - (\E \mtx{Y}_i)(\E \mtx{Z}_{i-1}) = (\E \mtx{Y}_i)(\mtx{Z}_{i-1} - \E \mtx{Z}_{i-1}) + (\mtx{Y}_i-\E \mtx{Y}_i)\mtx{Z}_{i-1}.
\end{equation}
As in~\eqref{eqn:fluctuation-mean}, the second term is a fluctuation that
is conditionally zero mean.  The fluctuation term satisfies
the norm bound~\eqref{eqn:fluctuation-bd}.

\subsection{Growth and Concentration}

Our main result controls the growth of the moments of
a product of independent random matrices.  It also describes
how well the random product concentrates around its
expectation.

\begin{theorem}[Growth and Concentration of Products] \label{thm:pnorm}
Consider a fixed matrix $\mtx{Z}_0 \in \CC^{d\times r}$ and an independent family $\{ \mtx{Y}_1,  \mtx{Y}_2, \dots, \mtx{Y}_n  \} \subset \M_d$
of random matrices.  Form the product
\[
\mtx{Z}_n = \mtx{Y}_n \mtx{Y}_{n-1} \cdots \mtx{Y}_2 \mtx{Y}_1 \mtx{Z}_0 \in \CC^{d \times r}.
\]
For parameters $2 \leq q \leq p$, assume that
\[
\norm{ \E \mtx{Y}_i } \leq m_i
\quad\text{and}\quad
\big( \E \norm{ \mtx{Y}_i - \E \mtx{Y}_i }^q \big)^{1/q}
	\leq \sigma_i m_i 
	\quad\text{for $i = 1, \dots, n$.}
\]
Define the product of means and the accumulated relative variance
\[
M = \prod_{i=1}^n m_i
\quad\text{and}\quad
v = \sum_{i=1}^n \sigma_i^2.
\]
Then the random product $\mtx{Z}_n$ satisfies the growth bound and the concentration bound
\begin{align}
\triplenorm{ \mtx{Z}_n }_{p,q} &\leq \econst^{ \cnst{C}_p v / 2}  \norm{ \mtx{Z}_0 }_{p} \cdot M; \label{eq:master_norm1} \\
\triplenorm{ \mtx{Z}_n - \E \mtx{Z}_n }_{p,q} &\leq \left( \econst^{\cnst{C}_p v} - 1 \right)^{1/2} \norm{ \mtx{Z}_0 }_{p} \cdot M. \label{eq:master_norm2}
\end{align}
\end{theorem}

\begin{proof}[Proof of Theorem~\ref{thm:pnorm}, relation \eqref{eq:master_norm1}]
By the homogeneity of~\eqref{eq:master_norm1}, we may assume that $m_i = 1$ for each index $i$,
so that also $M = 1$.  As in \eqref{eqn:prod-decomp}, we have the decomposition
\[
\mtx{Z}_i := \mtx{Y}_i\mtx{Z}_{i-1} = (\E \mtx{Y}_i)\mtx{Z}_{i-1} + (\mtx{Y}_i-\E \mtx{Y}_i)\mtx{Z}_{i-1}
	\quad\text{for each $i = 1, \dots, n$.}
\]
Now, Proposition~\ref{prop:Ricard_Xu} implies that
\begin{align*}
\triplenorm{\mtx{Z}_i }_{p,q}^2 &\leq \triplenorm{(\E \mtx{Y}_i)\mtx{Z}_{i-1}}_{p,q}^2 + \cnst{C}_p \cdot \triplenorm{(\mtx{Y}_i-\E \mtx{Y}_i) \mtx{Z}_{i-1}}_{p,q}^2\\
&\leq \norm{ \E \mtx{Y}_i }^2 \cdot\triplenorm{\mtx{Z}_{i-1}}_{p,q}^2 + \cnst{C}_p \left( \E \norm{\mtx{Y}_i-\E \mtx{Y}_i}^q \right)^{2/q} \cdot \triplenorm{\mtx{Z}_{i-1}}_{p,q}^2 \\
& \leq (1 + \cnst{C}_p \sigma_i^2) \cdot \triplenorm{\mtx{Z}_{i-1}}_{p,q}^2 \\
& \leq \exp(\cnst{C}_p \sigma_i^2)\cdot \triplenorm{\mtx{Z}_{i-1}}_{p,q}^2.
\end{align*}
The second line follows from~\eqref{eqn:fluctuation-bd}, and the third depends on our hypotheses about the factors $\mtx{Y}_i$. 
The last relation requires the numerical inequality $1 + a \leq \econst^a$, valid for all $a \in \R$.
By iteration, 
\begin{equation} \label{eqn:growth-i}
\triplenorm{\mtx{Z}_i }_{p,q}^2 \leq \exp\left( \cnst{C}_p \sum_{k=1}^i \sigma_k^2 \right) \cdot \norm{\mtx{Z}_{0}}_{p}^2.
\end{equation}
In the final step, we use the assumption that $\mtx{Z}_0$ is not random to see that $\triplenorm{\mtx{Z}_0}_{p,q} = \norm{\mtx{Z}_0}_p$.
For $i = n$, the formula~\eqref{eqn:growth-i} is the advertised result.
\end{proof}

\begin{proof}[Proof of Theorem~\ref{thm:pnorm}, relation \eqref{eq:master_norm2}]
The pattern of argument is similar with the proof of~\eqref{eq:master_norm1}.
By the homogeneity of~\eqref{eq:master_norm2}, we may assume that all $m_i = 1$ and that $M = 1$.
As in~\eqref{eqn:conc-decomp}, we have the decomposition
\[
\mtx{Z}_i - \E \mtx{Z}_i = \mtx{Y}_i\mtx{Z}_{i-1} - (\E \mtx{Y}_i)(\E \mtx{Z}_{i-1}) = (\E \mtx{Y}_i)(\mtx{Z}_{i-1} - \E \mtx{Z}_{i-1}) + (\mtx{Y}_i-\E \mtx{Y}_i)\mtx{Z}_{i-1}.
\]
Again, we invoke Proposition~\ref{prop:Ricard_Xu} to ascertain that
\begin{align*}
\triplenorm{\mtx{Z}_i - \E \mtx{Z}_i}_{p,q}^2
	&\leq \triplenorm{(\E \mtx{Y}_i)(\mtx{Z}_{i-1}-\E \mtx{Z}_{i-1})}_{p,q}^2 + \cnst{C}_p\cdot\triplenorm{(\mtx{Y}_i-\E \mtx{Y}_i) \mtx{Z}_{i-1}}_{p,q}^2\\
	&\leq \triplenorm{\mtx{Z}_{i-1}-\E \mtx{Z}_{i-1}}_{p,q}^2 + \cnst{C}_p \sigma_i^2 \cdot \triplenorm{\mtx{Z}_{i-1}}_{p,q}^2 \\
	&\leq \triplenorm{\mtx{Z}_{i-1}-\E \mtx{Z}_{i-1}}_{p,q}^2 + \cnst{C}_p \sigma_i^2 \exp\left( \sum_{k=1}^{i-1} \cnst{C}_p  \sigma_i^2 \right) \cdot \norm{ \mtx{Z}_{0} }_{p}^2.
\end{align*}
The last inequality is our growth bound~\eqref{eqn:growth-i}.  This recurrence relation delivers
\begin{align*}
\triplenorm{\mtx{Z}_n - \E \mtx{Z}_n}_{p,q}^2
	& \leq \triplenorm{\mtx{Z}_0 - \E \mtx{Z}_0}_{p,q}^2 + \left[ \sum_{i=1}^n \cnst{C}_p \sigma_i^2 \exp\left( \sum_{k=1}^{i-1} \cnst{C}_p \sigma_k^2 \right) \right] \cdot \norm{ \mtx{Z}_{0} }_{p}^2 \\
	& = \left[ \sum_{i=1}^n \cnst{C}_p \sigma_i^2 \exp\left( \sum_{k=1}^{i-1}\cnst{C}_p \sigma_k^2 \right)\right] \cdot \norm{ \mtx{Z}_{0} }_{p}^2 \\
	& \leq  \left[ \exp\left( \sum_{i=1}^n \cnst{C}_p \sigma_i^2 \right) -1 \right]  \cdot \norm{ \mtx{Z}_{0} }_{p}^2.
\end{align*}
The equality holds because $\mtx{Z}_0$ is not random.  The last relation is a numerical inequality,
whose proof appears in Lemma~\ref{lem:number_inequality}.
\end{proof}

Observe that the difference between the bounds~\eqref{eq:master_norm1} and~\eqref{eq:master_norm2}
is only visible when $\cnst{C}_p v$ is small, in which case
\begin{equation} \label{eqn:thm-upshot}
\econst^{\cnst{C}_p v/2} \approx 1
\quad\text{and}\quad
\left(\econst^{\cnst{C}_p v} - 1 \right)^{1/2} \approx \sqrt{\cnst{C}_p v}.
\end{equation}
This is the setting where the concentration result may be nontrivial.

The next two remarks contain some minor extensions of Theorem~\ref{thm:pnorm}.
Similar extensions are possible at other points in this paper.  For the most
part, we omit these developments.

\begin{remark}[Growth from Concentration]
In some instances, we can improve over the growth bound~\eqref{eq:master_norm1}
by applying the triangle inequality to the decomposition
$\mtx{Z}_n = (\E \mtx{Z}_n) + (\mtx{Z}_n - \E \mtx{Z}_n)$
and invoking the concentration bound~\eqref{eq:master_norm2}:
$$
\triplenorm{ \mtx{Z}_n }_{p,q}
	\leq \norm{ \E \mtx{Z}_n }_p + \left( \econst^{\cnst{C}_p v} - 1 \right)^{1/2} \norm{ \mtx{Z}_0 }_{p} \cdot M.
$$
Similarly, we can apply Proposition~\ref{prop:Ricard_Xu} together with~\eqref{eq:master_norm2} to obtain
$$
\triplenorm{ \mtx{Z}_n }_{p,q}^2
	\leq \norm{ \E \mtx{Z}_n }_p^2 + \cnst{C}_p \left( \econst^{\cnst{C}_p v} - 1 \right) \norm{\mtx{Z}_0}_p^2 \cdot M^2.
$$
Neither of these bounds represents a strict improvement over the other or over the growth bound~\eqref{eq:master_norm1}.
\end{remark}

\begin{remark}[Uniform Bounds on Factors] \label{rmk:strong_assumption}
Potentially stronger estimates are possible if the factors are bounded in norm.
Fix parameters $2 \leq q \leq p$.
Suppose that $\norm{ \mtx Y_i } \leq b_i$  almost surely and $\triplenorm{ \mtx{Y}_i - \E \mtx{Y}_i }_{p,q} \leq \sigma_i b_i$ for each index $i$.
Define $B = \prod_{i=1}^n b_i$ and $v = \sum_{i=1}^n \sigma_i^2$. Then  
\begin{align}
\triplenorm{\mtx{Z}_n}_{p,q} & \leq \norm{ \mtx{Z}_0}_p \cdot B; \label{eq:as_norm1} \\
\triplenorm{\mtx{Z}_n - \E \mtx{Z}_n }_{p,q} & \leq \sqrt{\cnst{C}_p v} \, \norm{ \mtx{Z}_0 }_p \cdot B. \label{eq:as_norm2}
\end{align}
Compare these results with \eqref{eq:master_norm1}, \eqref{eq:master_norm2}, and \eqref{eqn:thm-upshot}.
As for the proof, the growth bound~\eqref{eq:as_norm1} is an immediate consequence of the definition $\mtx{Z}_n = \mtx{Y}_n \cdots \mtx{Y}_1 \mtx{Z}_{0}$.
The concentration result~\eqref{eq:as_norm2} follows if we repeat the proof of~\eqref{eq:master_norm2},
using the growth bound~\eqref{eq:as_norm1} in place of~\eqref{eq:master_norm1}.
\end{remark}

\subsection{Expectation Bounds for the Spectral Norm}

In many cases, we just need to know the expected value of the
product $\norm{ \mtx{Z}_n }$ or the expected value of the
fluctuation $\norm{ \mtx{Z}_n - \E \mtx{Z}_n }$.  We can
obtain bounds for these quantities
as an easy consequence of Theorem~\ref{thm:pnorm}.

\begin{corollary}[Expectation Bounds] \label{cor:expectation}
Consider an independent sequence $\{ \mtx{Y}_1, \dots, \mtx{Y}_n \} \subset \M_d$
of random matrices, and form the product $\mtx{Z}_n = \mtx{Y}_n \cdots \mtx{Y}_1$.
Assume that
\[
\norm{ \Expect \mtx{Y}_i } \leq m_i
\quad\text{and}\quad
\big( \Expect \norm{ \mtx{Y}_i - \Expect \mtx{Y}_i }^2 \big)^{1/2} \leq \sigma_i m_i
\quad\text{for $i = 1, \dots, n$.}
\]
Let $M = \prod_{i=1}^n m_i$ and $v = \sum_{i=1}^n \sigma_i^2$.  Then
\begin{align}
\E \norm{ \mtx{Z}_n }
	&\leq \exp\left( \sqrt{2 v \, (2 v \vee \log d)} \right) \cdot M. 
	\label{eqt:expectation_1}
\intertext{Provided that $v\,(1+ 2 \log d) \leq 1$, then also}
\E \norm{ \mtx{Z}_n - \E \mtx{Z}_n }
	&\leq \sqrt{\econst^2 v \, (1+ 2 \log d) } \cdot M. 
	\label{eqt:expectation_2}
\end{align}
\end{corollary}

\begin{proof}
To apply Theorem~\ref{thm:pnorm}, we set $\mtx{Z}_0 = \Id$ and choose the power $q = 2$.

To obtain the growth bound~\eqref{eqt:expectation_1},
consider the Schatten norm of order $p = \sqrt{2 (2v \vee \log d) / v}$.
Note that $p \geq 2$ and that $\norm{ \mtx{Z}_0 }_p \leq d^{1/p} \leq \econst^{pv/2}$.
Invoke Theorem~\ref{thm:pnorm}, relation~\eqref{eq:master_norm1},
to see that
\[
\E \norm{ \mtx{Z}_n }
	\leq \triplenorm{ \mtx{Z}_n }_{p,2}
	\leq \econst^{\cnst{C}_p v/2} \norm{ \mtx{Z}_0 }_p \cdot M
	\leq \econst^{p v/2} \cdot \econst^{pv/2} \cdot M
	= \econst^{pv} \cdot M.
\]
We used the fact that $\cnst{C}_p = p - 1 < p$.  This is the stated result.

To obtain the concentration bound~\eqref{eqt:expectation_2},
consider the Schatten norm $p = 2(1 + \log d)$.
Note that $p \geq 2$ and that $\norm{ \mtx{Z}_0 }_p \leq d^{1/p} \leq \sqrt{\econst}$.
Now, we use Theorem~\ref{thm:pnorm}, relation~\eqref{eq:master_norm2},
in a similar fashion.  Assuming that $\cnst{C}_p v \leq 1$,
\[
\E \norm{ \mtx{Z}_n - \E \mtx{Z}_n }
	\leq \triplenorm{ \mtx{Z}_n - \E \mtx{Z}_n }_{p,2}
	\leq \big( \econst^{\cnst{C}_p v} - 1 \big)^{1/2} \norm{\mtx{Z}_0}_p \cdot M
	\leq \econst \sqrt{ \cnst{C}_pv } \cdot M. 
\]
The last bound is the numerical inequality $\econst^a - 1 \leq \econst a$,
valid when $a \in [0,1]$.  Finally, note that
$\cnst{C}_p = p - 1 = 1 + 2 \log d$.
\end{proof}

The inequality \eqref{eqt:expectation_1} shows its power when each $\sigma_i$ is small.
Assume that each $m_i = 1$ and $\sigma_i \leq b/n$ for a constant $b$.
Then it is not hard to check that
\[
\norm{\E\mtx{Z}_n} \leq 1
\quad \text{and}\quad
\norm{ \mtx{Z}_n } \leq \left(1 + (b/n) \right)^n\leq \econst^b.
\]
If $L\sqrt{(2 \log d)/n}$ is close to zero, then
\eqref{eqt:expectation_1} implies
\[
\E \norm{ \mtx{Z}_n } \leq \econst^{ b \sqrt{(2\log d)/n} } \approx 1.
\]
That is, $\E \norm{ \mtx{Z}_n }$ is much closer to $\norm{ \E\mtx{Z}_n }$
than to the worst-case value $\econst^b$.

\begin{remark}[Uniform Bounds on Factors] \label{rmk:strong_assumption_2}
Fix $p \geq 2$. Assume that $\norm{ \mtx{Y}_i } \leq b_i$ almost surely
and $\triplenorm{ \mtx{Y}_i - \E \mtx{Y}_i }_{p,2} \leq \sigma_i b_i$ for each $i$.
Let $v = \sum_{i=1}^n \sigma_i^2$ and $B = \prod_{i=1}^n b_i$.  Then Remark~\ref{rmk:strong_assumption} implies that
\[
\E \norm{ \mtx{Z}_n - \E \mtx{Z}_n } \leq \sqrt{\econst v \, (1+ 2\log d)} B.
\]
This improves the constant in \eqref{eqt:expectation_2} by a factor of $\sqrt{\econst}$,
and it removes the condition that $v \, (1+ 2 \log d) \leq 1$.
\end{remark}

\subsection{Tail Bounds for the Spectral Norm}

The moment bounds in Theorem~\ref{thm:pnorm} can also be upgraded
to obtain tail bounds for $\norm{\mtx{Z}_n}$ and $\norm{\mtx{Z}_n - \E \mtx{Z}_n}$.

\begin{corollary}[Tail Bounds] \label{cor:concentration}
Consider an independent sequence $\{ \mtx{Y}_1, \dots, \mtx{Y}_n \} \subset \M_d$
of random matrices, and form the product $\mtx{Z}_n = \mtx{Y}_n \cdots \mtx{Y}_1$.
Assume that
\[
\norm{ \Expect \mtx{Y}_i } \leq m_i
\quad\text{and}\quad
\norm{ \mtx{Y}_i - \Expect \mtx{Y}_i } \leq \sigma_i m_i
\quad\text{almost surely for $i = 1, \dots, n$.}
\]
Let $M = \prod_{i=1}^n m_i$ and $v = \sum_{i=1}^n \sigma_i^2$.
Then
\begin{align} 
\Prob{ \norm{ \mtx{Z}_n } \geq t M}
	&\leq d \cdot \exp\left(\frac{-\log^2 t}{2v}\right)
	\quad\text{when $\log t \geq 2v$.} \label{eqn:tail-bound-large}
\intertext{Furthermore,}
\Prob{ \norm{ \mtx{Z}_n - \E \mtx{Z}_n } \geq t M}
	&\leq (d \vee \econst) \cdot \exp\left(\frac{-t^2}{2\econst^2v}\right)
	\quad\text{when $t \leq \econst$.} \label{eqn:tail-bound-small} \\
\end{align}
\end{corollary}

\begin{proof}
We begin with the proof of~\eqref{eqn:tail-bound-large}.
By homogeneity, we may assume that $m_i = 1$ for each $i$, so also $M = 1$.
Apply Markov's inequality and~\eqref{eqn:moment-method} to obtain
\[
\Prob{ \norm{ \mtx{Z}_n } \geq t }
	\leq \inf_{p \geq 2}\ t^{-p}\cdot \E\norm{\mtx{Z}_n}^p
	\leq \inf_{p \geq 2}\ t^{-p}\cdot \triplenorm{ \mtx{Z}_n }_{p,p}^p.
\]
To bound the $L_p(S_p)$ norm, we will use Theorem~\ref{thm:pnorm}
with $\mtx{Z}_0 = \Id$ and with $q = p$.  Relation~\eqref{eq:master_norm1}
gives
\[
t^{-p} \cdot \triplenorm{ \mtx{Z}_n }_{p,p}^p
	\leq t^{-p} \cdot \econst^{p \cnst{C}_p v / 2} \norm{ \mtx{Z}_0 }_p^p
	= d \cdot \big( t^{-2} \econst^{\cnst{C}_p v} \big)^{p/2}.
\]
We have used the fact that $\norm{\mtx{Z}_0}_p^p = \norm{\Id}_p^p = d$.
Under the assumption that $\log t \geq 2v$,
we may select $p = (\log t) / v \geq 2$.  This choice yields
\[
d \cdot \big( t^{-2} \econst^{pv} \big)^{p/2}
	= d \cdot \exp\left( \frac{-\log^2 t}{2v} \right).
\]
Sequence the last three displays to arrive at the bound~\eqref{eqn:tail-bound-large}.

We establish~\eqref{eqn:tail-bound-small} in an analogous fashion.
The same argument, using relation~\eqref{eq:master_norm2}, implies that
\[
\Prob{ \norm{ \mtx{Z}_n - \E \mtx{Z}_n } \geq t }
	\leq \inf_{p \geq 2}\ d \cdot \left[ t^{-2} \big( \econst^{\cnst{C}_p v} - 1 \big) \right]^{p/2}.
\]
Supposing that $t^2 / (\econst^2 v) < 2$,
the bound~\eqref{eqn:tail-bound-small} holds trivially because $\econst \cdot \exp(-t^2/(2\econst^2 v)) \geq 1$.
Otherwise, we may select the parameter $p = t^2 / (\econst^2 v) \geq 2$.
Under the assumption that $t \leq \econst$,
$\cnst{C}_p v \leq pv \leq (t/\econst)^2 \leq 1$, so that
$\econst^{\cnst{C}_p v} - 1 \leq \econst \cnst{C}_p v \leq t^2/\econst$.
Therefore,
\[
d \cdot \big[ t^{-2} \big( \econst^{\cnst{C}_p v} - 1 \big) \big]^{p/2}
	\leq d \cdot \econst^{-p/2}
	= d \cdot \exp \left( \frac{-t^2}{2\econst^2 v} \right).
\]
The last two displays imply~\eqref{eqn:tail-bound-small}.
\end{proof}

\begin{remark}[Uniform Bounds on Factors]
In the setting of Remark~\ref{rmk:strong_assumption_2},
we have an unconditional variant of the concentration bound~\eqref{eqn:tail-bound-small}:
\begin{equation}
\Prob{ \norm{ \mtx{Z}_n - \E \mtx{Z}_n } \geq t \cdot B}
	\leq (d \vee \econst) \cdot \exp\left(\frac{-t^2}{2\econst v}\right)
	\quad\text{for all $t > 0$.}
\end{equation}
\end{remark}

\section{Application: Random Perturbations of the Identity}
\label{sec:perturbations}

This section treats the fundamental case where the factors $\mtx{Y}_i$ in the product
are independent, random perturbations of the identity.
That is, $\mtx{Y}_i = \Id + \mtx{X}_i$ where $\{ \mtx{X}_i \} \subset \M_d$ is an independent family.
We will develop specialized theory for this class of problems,
and we will use these results to compare our work with several recent papers.

\subsection{Iterative Algorithms}

To motivate this development, observe that random perturbations of the identity
arise from the analysis of the iterative scheme 
\begin{equation}\label{eqt:iterative_scheme}
\vct{u}^{(i+1)} = \vct{u}^{(i)} + \mtx{X}_i\vct{u}^{(i)}
\quad\text{for $i=1,2, 3\dots$.}
\end{equation}
where $\mtx{X}_i\vct{u}^{(i)}$ is a linear update to the current iterate $\vct{u}^{(i)}$.
In this application, the norm of each $\mtx{X}_i$ is proportional to the step size of the scheme,
so it is typically small and it is controlled by the user.
For example, the updates in Oja's algorithm~\cite{Oja82:Simplified-Neuron}
take the form~\eqref{eqt:iterative_scheme}.

For now, we do not permit the random matrix $\mtx{X}_i$ to depend on the sequence $\{ \vct{u}^{(i)} \}$ of iterates.
Later, in Section~\ref{sec:adapted}, we describe an extension of our approach to
the setting where $\{ \mtx{X}_i \}$ is an adapted sequence.  This variant
allows for the study of a wider class of iterative algorithms.

\subsection{Bounds for the Product}

First, we develop bounds for the growth and concentration of
a product of perturbations of the identity.
In Section~\ref{sec:inverses}, we develop results for the inverse of the product.

\begin{corollary}[Perturbations of the Identity] \label{cor:expectation_Oja}
Consider an independent family $\{ \mtx{X}_1, \dots, \mtx{X}_n \} \subset \M_d$ of random matrices,
and form the product $\mtx{Z}_n = (\Id + \mtx{X}_n) \cdots (\Id + \mtx{X}_1)$.
Assume that
\[
\norm{ \E \mtx{X}_i } \leq \xi_i
\quad\text{and}\quad
\norm{ \mtx{X}_i - \E \mtx{X}_i } \leq \sigma_i
\quad\text{almost surely for $i = 1, \dots, n$.}
\]
Define $\xi = \sum_{i=1}^n \xi_i$ and $v = \sum_{i=1}^n \sigma_i^2$.  Then
\begin{align}
\E \norm{\mtx{Z}_n} &\leq \exp\left( \xi + \sqrt{2 v \log d} \right)
	&&\text{when $2v \leq \log d$;}
	\label{eqt:expectation_Oja_1} \\
\E \norm{\mtx{Z}_n - \E \mtx{Z}_n} &\leq  \econst^{\xi + 1} \sqrt{v\, (1 + 2 \log d)}
	&&\text{when $v \, (1 + 2\log d) \leq 1$.}
	\label{eqt:expectation_Oja_2}
\intertext{Moreover,}
\Prob{ \norm{\mtx{Z}_n} \geq t \econst^{\xi} }
	&\leq d \cdot \exp\left( \frac{-\log^2 t}{2v} \right)
	&&\text{when $\log t \geq 2v$};
	\label{eqt:tail_Oja_1} \\
\Prob{ \norm{\mtx{Z}_n - \E \mtx{Z}_n} \geq t \econst^{\xi} }
	&\leq (d \vee \econst) \cdot \exp\left(\frac{-t^2}{2\econst^2 v}\right)
	&&\text{when $t \leq \econst$.}
	\label{eqt:tail_Oja_2}
\end{align}
\end{corollary}

\begin{proof}
Let $\mtx{Y}_i = \Id+\mtx{X}_i$ for each index $i$.  Then
\[
\norm{ \E \mtx{Y}_i } \leq 1 + \norm{ \E \mtx{X}_i } \leq \econst^{\xi_i} =: m_i.
\]
Furthermore, since $m_i \geq 1$,
\[
\norm{ \mtx{Y}_i-\E \mtx{Y}_i } = \norm{ \mtx{X}_i-\E \mtx{X}_i } \leq \sigma_i \leq \sigma_i m_i.
\]
The results follow instantly from Corollary \ref{cor:expectation} and Corollary \ref{cor:concentration}.
\end{proof}

\subsection{Comparison with Prior Work}

To clarify the meaning of Corollary~\ref{cor:expectation_Oja},
let us elaborate what it predicts when
\begin{equation*} \label{eqn:perturb-example}
\norm{ \E \mtx{X}_i } \leq T/n
\quad\text{and}\quad
\norm{ \mtx{X}_i - \E \mtx{X}_i } \leq L/n
\quad\text{for constants $T, L$.}
\end{equation*}
This situation can arise if we perform $n$ iterations of
the iterative scheme~\eqref{eqt:iterative_scheme} with a
uniform step size of $1/n$.
In this setting, Corollary~\ref{cor:expectation_Oja} implies that
\begin{equation} \label{eqt:Henriksen_Ward_1}
\E \norm{ \mtx{Z}_n - \E \mtx{Z}_n } \leq \sqrt{\frac{1 + 2\log d}{n}} L \econst^{1 + T}
\quad\text{when $L^2 (1 + 2 \log d) \leq n$.}
\end{equation}
For $\delta \in [0,1]$, with probability at least $1 - \delta$,
\begin{equation} \label{eqt:Henriksen_Ward_2}
\norm{ \mtx{Z}_n - \E \mtx{Z}_n } \leq \sqrt{\frac{2 + 2\log(d/\delta)}{n}} L \econst^{1 + T}
\quad\text{when $L^2 (2 + 2 \log (d/\delta)) \leq n$.}
\end{equation}
Furthermore, if we assume that $\norm{ \mtx{X}_i } \leq T/n$ almost surely for each $i$,
then Remark~\ref{rmk:strong_assumption_2} implies that \eqref{eqt:Henriksen_Ward_1}
and \eqref{eqt:Henriksen_Ward_2} hold without restriction.

The paper~\cite{henriksen2019concentration} of Henriksen and Ward only contemplates
the situation described in the last paragraph.  It obtains a concentration bound of
the form
$$
\norm{ \mtx{Z}_n - \E \mtx{Z}_n } \leq \frac{L \econst^L}{\sqrt{n}} \cdot \mathrm{polylog}(n,d,1/\delta)
\quad\text{with probability at least $1 - \delta$.}
$$
The salient improvement in~\eqref{eqt:Henriksen_Ward_2} stems from the reduction of the factor
$\econst^L$ to $\econst^T$.
This difference is most pronounced when $\E \mtx{X}_i = 0$ for each $i$, in which case the bound~\eqref{eqt:Henriksen_Ward_2} removes the exponential factor entirely.
Even under the assumption that $\mtx{X}_i \psdge \mtx{0}$ for all each $i$, it can happen that $L \geq dT$, so this refinement can make a big difference.

Last, we mention one instance that has special importance.  Let $\mtx{A} \in \M_d$ be a fixed
matrix.  Consider a triangular array $\{ \mtx{X}_i^{(n)} : \text{$i \leq n$ and $n \in \N$} \} \subset \M_d$
of independent random matrices.  For each index $n$, assume that
$$
\Expect \mtx{X}_i^{(n)} = \mtx{A} / n
\quad\text{and}\quad
\norm{ \smash{ \mtx{X}_i^{(n)} - \E \mtx{X}_i^{(n)} } } \leq L/n
\quad\text{for $i = 1, \dots, n$.}
$$
Define the product
$$
\mtx{Z}^{(n)} = \big(\Id + \mtx{X}_n^{(n)} \big) \cdots \big(\Id + \mtx{X}_1^{(n)}\big).
$$
By functional calculus,
$$
\Expect \mtx{Z}^{(n)} = (\Id + \mtx{A} / n)^n \to \econst^{\mtx{A}}
\quad\text{as $n \to \infty$.} 
$$
The bound~\eqref{eqt:Henriksen_Ward_2}, combined with the first Borel--Cantelli Lemma,
guarantees that
$$
\mtx{Z}^{(n)} \to \econst^{\mtx{A}}
\quad\text{as $n \to \infty$, almost surely.}
$$
This result is a special case of the limit theorem of Emme and Hubert~\cite[Thm.~1.1]{emme2017limit}.
They do not require independence, but they only achieve an asymptotic result.
Our analysis gives a rate of convergence that matches the corresponding bound~\eqref{eq:concentration_inf} for scalar random variables.

\subsection{Bounds for the Inverse of a Product}
\label{sec:inverses}

In some applications, it is valuable to have a lower bound for the minimum
singular value of a random product.  Equivalently, we can seek an upper bound
for the spectral norm of the inverse of the product.  This section describes
a situation where clean results are possible.

Consider the case where the factors $\mtx{Y}_i$ are
perturbations of the identity: $\mtx{Y}_i = \Id + \mtx{X}_i$,
where $\mtx{X}_i$ is small enough to ensure that $\mtx{Y}_i$
is invertible with probability $1$.  In this setting,
we can easily study the inverse of the product
using Corollary~\ref{cor:expectation_Oja}.

\begin{corollary}[Perturbations of the Identity: Inverses] \label{cor:expectation_Oja_inverse}
Frame the same hypotheses as in Corollary \ref{cor:expectation_Oja}.
Assume that $\xi_i+\sigma_i<1$ for each index $i$, and define
\[
\bar{\xi} = \sum_{i=1}^n \left[ \xi_i + \frac{(\xi_i + \sigma_i)^2}{1 - (\xi_i + \sigma_i)} \right]
\quad\text{and}\quad
\bar{v} = \sum_{i=1}^n \left[ \sigma_i + \frac{2 (\xi_i + \sigma_i)^2}{1 - (\xi_i + \sigma_i)} \right]^2.
\]
Then
\begin{align*}
\E \norm{\mtx{Z}_n^{-1}} &\leq \exp\left( \bar{\xi} + \sqrt{2 \bar{v} \log d} \right)
	&&\text{when $2 \bar{v} \leq \log d$;} \\
\E \norm{\mtx{Z}_n^{-1} - \E \mtx{Z}_n^{-1}} &\leq  \econst^{\bar{\xi}} \sqrt{\econst^2 \bar{v}\, (1 + 2 \log d)}
	&&\text{when $\bar{v} \, (1 + 2\log d) \leq 1$.}
\end{align*}
\end{corollary}

\begin{proof}
With the same notation as in Corollary~\ref{cor:expectation_Oja},
observe that $\mtx{Z}_n^{-1} = (\Id + \mtx{X}_1)^{-1} \cdots (\Id + \mtx{X}_n)^{-1}$.
This is an independent product that can be bounded by applying the corollary.
To do so, we simply need to express $(\Id + \mtx{X}_i)^{-1} = \Id + \bar{\mtx{X}}_i$
for suitable random matrices $\bar{\mtx{X}}_i$.
The perturbation terms $\bar{\mtx{X}}_i$ are obtained from the calculation
\[
(\Id + \mtx{X}_i)^{-1} = \Id + \sum_{k=1}^\infty (-1)^k \mtx{X}_i^k
	= \Id - \mtx{X}_i + \mtx{X}_i^2 (\Id + \mtx{X}_i)^{-1}
	=: \Id + \bar{\mtx{X}}_i.
\]
It remains to develop estimates for the size of the perturbation.

The uniform bound $\norm{ \mtx{X}_i } \leq \norm{ \E \mtx{X}_i } + \norm{ \mtx{X}_i - \E \mtx{X}_i } \leq \xi_i + \sigma_i < 1$ implies that
\[
\norm{ (\Id + \mtx{X}_i)^{-1} } \leq \big( 1 - \norm{\mtx{X}_i} \big)^{-1} \leq \frac{1}{1 - (\xi_i + \sigma_i)}.
\]
Therefore, the norm of the expected perturbation satisfies
\[
\norm{ \E \bar{\mtx{X}}_i } \leq \norm{ \E \mtx{X}_i } + \norm{ \E \big[ \mtx{X}_i^2 (\Id + \mtx{X}_i)^{-1} \big] }
	\leq \xi_i + \frac{(\xi_i + \sigma_i)^2}{1 - (\xi_i + \sigma_i)}
	=: \bar{\xi}_i.
\]
The fluctuations of the perturbation satisfy
\[
\norm{ \bar{\mtx{X}_i} - \E \bar{\mtx{X}_i} }
	\leq \norm{ \mtx{X}_i - \E \mtx{X}_i } + 2 \norm{ \mtx{X}_i^2 (\Id + \mtx{X}_i)^{-1} }
	\leq \sigma_i + \frac{2(\xi_i + \sigma_i)^2}{1 - (\xi_i + \sigma_i)}
	=: \bar{\sigma}_i.
\]
The results follow when we apply Corollary~\ref{cor:expectation_Oja}
with the random matrices $\bar{\mtx{X}}_i$ in place of the $\mtx{X}_i$.
\end{proof}

\section{Improvements and Extensions}
\label{sec:extensions}

The argument underlying Theorem~\ref{thm:pnorm} has several natural extensions.
First, we develop sharper results for products of random contractions.
In Section~\ref{sec:compression}, we derive better estimates for a
matrix product where the initial term is rectangular.
In Section~\ref{sec:adapted}, we document the changes that are
necessary in case the factors in the product are not independent
but form an adapted sequence.  Last, In Section~\ref{sec:spectral-radius},
we explain how to develop a bound on the spectral radius of a product.

\subsection{A Product of Contractions}

Most of our results are designed for products of general random matrices.
In some circumstances, the factors in the product are \emph{contractions},
matrices whose singular values are bounded by one.  For example, the
randomized Kaczmarz algorithm~\cite{SV09:Randomized-Kaczmarz} can be
expressed as the repeated application of random contractions.  Other
randomized linear fixed-point iterations take a similar form.  This
section derives sharper estimates for this important setting.

\begin{theorem}[Product of Contractions] \label{thm:contract}
Consider an independent family $\{ \mtx{Y}_1, \dots, \mtx{Y}_n \} \subset \M_d$
of random contractions; that is, $\norm{ \mtx{Y}_i } \leq 1$.  Form the
random product $\mtx{Z}_n = \mtx{Y}_n \cdots \mtx{Y}_1$.
Assume that
$$
\norm{ \E \abs{\mtx{Y}_i}^2 } \leq m_i^2 \leq 1
\quad\text{and}\quad
\norm{ \mtx{Y}_i - \E \mtx{Y}_i } \leq \sigma_i m_i 
\quad\text{almost surely for $i = 1, \dots, n$.}
$$
Define $M := \prod_{i=1}^n m_i$ and $v := \sum_{i=1}^n \sigma_i^2$.  Then 
\begin{align}
\E \norm{ \mtx{Z}_n } &\leq 1 \wedge (\sqrt{d} \cdot M ); \label{eqn:contract-mean}\\
\E \norm{ \mtx{Z}_n - \E \mtx{Z}_n } &\leq \sqrt{dv} \cdot M. 
\label{eqn:contract-conc}
\end{align}
Furthermore, we have the tail bound 
\begin{equation} \label{eqn:contract-tail}
\Prob{ \norm{ \mtx{Z}_n - \E \mtx{Z}_n } \geq t }
	\leq dM^2 \cdot \econst^{-t^2 /(2\econst v)}
	\quad\text{when $t^2 \geq 2\econst v$.}
\end{equation}
\end{theorem}

To prove this result, we require a lemma that isolates
the influence of each factor in the product.  This step exploits
the uniform bound on the singular values in an essential way.

\begin{lemma}[Random Contractions] \label{lem:contractions}
Let $\mtx{Y} \in \M_d$ be a random contraction, and let $\mtx{Z} \in \M_d$
be a random matrix that is independent from $\mtx{Y}$.  For $2 \leq q \leq p$,
$$
\triplenorm{ \mtx{YZ} }_{p, q} \leq \norm{ \E \abs{ \mtx{Y} }^2 }^{1/p} \cdot \triplenorm{ \mtx{Z} }_{p,q}.
$$
\end{lemma}

\begin{proof}
Write out the $L_q(S_p)$ norm, and introduce matrix absolute values:
$$
\triplenorm{ \mtx{YZ} }_{p, q}^q = \E \norm{ \mtx{YZ} }_p^q
	= \E \left[ \trace \big( \mtx{Z}^* \mtx{Y}^* \mtx{Y} \mtx{Z} \big)^{p/2} \right]^{q/p}
	= \E \left[ \trace \Big( \abs{ \mtx{Z}^* } \cdot \abs{\mtx{Y}}^2 \cdot \abs{\mtx{Z}^*} \Big)^{p/2} \right]^{q/p}.
$$
The last relation can be verified using polar factorizations.
Apply the Araki--Lieb--Thirring inequality~\cite[Thm.~IX.2.20]{Bha97:Matrix-Analysis}
to distribute the power onto the factors in the trace.  We obtain
$$
\begin{aligned}
\triplenorm{ \mtx{YZ} }_{p, q}^q
	&\leq \E \left[ \trace\Big( \abs{\mtx{Z}^*}^{p/2} \cdot \abs{\mtx{Y}}^{p} \cdot \abs{\mtx{Z}^*}^{p/2} \Big) \right]^{q/p} \\
	&\leq \E_{\mtx{Z}} \E_{\mtx{Y}} \left[ \trace\Big( \abs{\mtx{Z}^*}^{p/2} \cdot \abs{ \mtx{Y} }^2 \cdot \abs{ \mtx{Z}^* }^{p/2} \Big) \right]^{q/p} \\
	&\leq \E_{\mtx{Z}} \left[ \trace\Big( \abs{ \mtx{Z}^* }^{p/2} \cdot \big(\E_{\mtx{Y}} \abs{\mtx{Y}}^2 \big) \cdot \abs{\mtx{Z}^*}^{p/2} \Big) \right]^{q/p}.
\end{aligned}
$$
The second inequality holds because a contraction satisfies $\abs{\mtx{Y}}^{p} \psdle \abs{ \mtx{Y} }^2$
for each $p \geq 2$.  The third inequality is Jensen's, justified because $q/p \leq 1$.
Bounding the matrix in the center by its norm,
$$
\triplenorm{ \mtx{YZ} }_{p, q}^q
	\leq \norm{ \E \abs{ \mtx{Y} }^2 }^{q/p} \cdot \E \left[ \trace \abs{ \mtx{Z}^* }^{p}  \right]^{q/p}
	= \norm{ \E \abs{ \mtx{Y} }^2 }^{q/p} \cdot \triplenorm{ \mtx{Z} }_{p,q}^q.
$$
This completes the analysis.
\end{proof}

With this result at hand, Theorem~\ref{thm:contract} follows from familiar arguments.

\begin{proof}[Proof of Theorem~\ref{thm:contract}]
Define $\mtx{Z}_0 = \Id$ and $\mtx{Z}_i = \mtx{Y}_i \mtx{Z}_{i-1}$ for each index $i = 1, \dots, n$.
We begin with the proof of~\eqref{eqn:contract-mean}.  
Since each factor is a contraction, it is clear that
$$
\E \norm{ \mtx{Z}_n } \leq \E \prod_{k=1}^n \norm{\mtx{Y}_k} \leq 1.
$$
To obtain a less trivial bound on the expectation,
we apply Lemma~\ref{lem:contractions} repeatedly.  For $p \geq 2$,
\begin{equation} \label{eqn:contract-mean-pf}
\E \norm{ \mtx{Z}_i } \leq \triplenorm{ \mtx{Z}_i }_{p,p}
	\leq \prod_{k=1}^i \norm{ \E \abs{\mtx{Y}_k}^2 }^{1/p} \cdot \triplenorm{ \Id }_{p,p}
	\leq d^{1/p} \prod_{k=1}^i m_k^{2/p}.
\end{equation}
The statement~\eqref{eqn:contract-mean} combines these two observations
when we set $i = n$ and $p = 2$.

Let us continue with the proof of~\eqref{eqn:contract-conc}, which is analogous
to the argument in Theorem~\ref{thm:pnorm}\eqref{eq:master_norm2}.
First, by expanding the inequality $\E \abs{\mtx Y_i - \E \mtx Y_i}^2 \psdge \mtx{0}$, we see that $\mtx{0} \psdle \abs{\E \mtx Y_i}^2 \psdle \E \abs{\mtx Y_i}^2$.
As a consequence,
$$
\norm{\E \mtx{Y}_i}^2 \leq \norm{ \E \abs{\mtx{Y}_i}^2} \leq m_i^2.
$$
For $p \geq 2$, calculate that
$$
\begin{aligned}
\triplenorm{ \mtx{Z}_i - \E \mtx{Z}_i }_{p,p}^2
	&\leq \triplenorm{ (\E \mtx{Y}_i) (\mtx{Z}_{i-1} - \E \mtx{Z}_{i-1} ) }_{p,p}^2 + \cnst{C}_p \cdot \triplenorm{ (\mtx{Y}_i - \E \mtx{Y}_i) \mtx{Z}_{i-1} }_{p,p}^2 \\
	&\leq m_i^{2} \cdot \triplenorm{ \mtx{Z}_{i-1} - \E \mtx{Z}_{i-1} }_{p,p}^2 + \cnst{C}_p \sigma_i^2 m_i^{2} \cdot \triplenorm{\mtx{Z}_{i-1}}_{p,p}^2 \\
	&\leq m_i^{4/p} \cdot \triplenorm{ \mtx{Z}_{i-1} - \E \mtx{Z}_{i-1} }_{p,p}^2 + \cnst{C}_p \sigma_i^2 \cdot  d^{2/p} \prod_{k=1}^{i} m_k^{4/p}.
\end{aligned}
$$
The second inequality is Lemma~\ref{lem:contractions}, and the third inequality requires~\eqref{eqn:contract-mean-pf}.
We have also used the fact that $m_i^{2} \leq m_i^{4/p}$ because $m_i \leq 1$.  Unrolling the recursion,
\begin{equation} \label{eqn:contract-recurse}
\triplenorm{ \mtx{Z}_n - \E \mtx{Z}_n }_{p,p}^2
	\leq \cnst{C}_p d^{2/p} \left( \prod_{i=1}^n m_i^{4/p} \right) \left( \sum_{i=1}^n \sigma_i^2 \right)
	= \cnst{C}_p d^{2/p} M^{4/p} v.
\end{equation}
For $p = 2$, this result implies the advertised bound~\eqref{eqn:contract-conc}.

Finally, the tail inequality~\eqref{eqn:contract-tail} follows from the estimate
$$
\Prob{ \norm{ \mtx{Z}_n - \E \mtx{Z}_n } \geq t } \leq \min_{p \geq 2}\ t^{-p} \cdot \triplenorm{ \mtx{Z}_n - \E \mtx{Z}_n }_{p,p}^p
	\leq (dM^2) \cdot \min_{p \geq 2}\ \left(\frac{pv}{t^2}\right)^{p/2}.
$$
The last inequality follows from~\eqref{eqn:contract-recurse} and $\cnst{C}_p < p$.
Bound the minimum with the power $p = t^2 / (\econst v) \geq 2$ to complete the argument.
\end{proof}

\subsection{Low-Rank Products}
\label{sec:compression}

So far, we have focused on the setting where the initial matrix $\mtx{Z}_0 = \Id$.
In many applications, we are interested in the action of the random product $\mtx{Y}_n \cdots \mtx{Y}_1 \in \M_d$
on a specific matrix $\mtx{Z}_0 \in \CC^{d \times r}$ with relatively few columns.
In this case, the terms that the control the behavior
of the product may be significantly smaller.
Here is an example of the kinds of results
one can achieve.

\begin{theorem}[Growth and Concentration of Low-Rank Products] \label{thm:compression_projection}
Consider a fixed matrix $\mtx{Z}_0 \in \CC^{d \times r}$ and an independent sequence
$\{ \mtx{Y}_1, \dots, \mtx{Y}_n \} \subset \M_d$ of random matrices.
Form the product $\mtx{Z}_n = \mtx{Y}_n \cdots \mtx{Y}_1 \mtx{Z}_0$.
Assume that
\[
\norm{ \E \mtx{Y}_i } \leq m_i
\quad\text{and}\quad
\sup_{\mtx{P} \in \mathcal{P}_r} \left( \E \norm{ (\mtx{Y}_i - \E \mtx{Y}_i) \mtx{P} }^2 \right)^{1/2} \leq \sigma_i m_i
\quad\text{for $i =1, \dots, n$,}
\]
where $\mathcal{P}_r \subset \M_d$ is the set of rank-$r$ orthogonal projectors.
Define $M = \prod_{i=1}^n m_i$ and $v = \sum_{i=1}^n \sigma_i^2$.  For each $p \geq 2$,
\begin{align}
\E \norm{ \mtx{Z}_n }
	&\leq \econst^{\cnst{C}_p v / 2} \cdot \norm{\mtx{Z}_0}_p \cdot M. 
	\label{eqt:expectation_projection_1} \\
\E \norm{ \mtx{Z}_n - \E \mtx{Z}_n }
	&\leq \big( \econst^{\cnst{C}_p v} - 1 \big)^{1/2} \cdot \norm{\mtx{Z}_0}_p \cdot M. 
	\label{eqt:expectation_projection_2} 
\end{align}
\end{theorem}

\begin{proof}
Define $\mtx{Z}_{i} = \mtx{Y}_i \mtx{Z}_{i-1}$ for each index $i$.  Since $\mtx{Z}_0 \in \CC^{d \times r}$,
the rank of each matrix $\mtx{Z}_i$ is at most $r$.  Thus, we can write $\mtx{Z}_i = \mtx{P}_i \mtx{Z}_i$,
where $\mtx{P}_i$ is a rank-$r$ orthogonal projector that only depends on $\mtx{Y}_{i}, \dots, \mtx{Y}_1$ and $\mtx{Z}_0$.
As a consequence,
\begin{align*}
\triplenorm{ (\mtx Y_i - \E \mtx Y_i) \mtx Z_{i-1} }_{p,2}
	&= \triplenorm{ (\mtx Y_i - \E \mtx Y_i) \mtx P_{i-1} \mtx Z_{i-1} }_{p,2} \\
	&\leq \big( \E \big[ \norm{  (\mtx Y_i - \E \mtx Y_i) \mtx P_{i-1} }^2 \cdot \norm{ \mtx{Z}_{i-1} }_p^2 \big] \big)^{1/2} \\
	& \leq \sup_{\mtx{P} \in \mathcal{P}_r} \big( \E \norm{ (\mtx Y_i - \E \mtx Y_i) \mtx{P} }^2 \big)^{1/2}
		\cdot \big( \E \norm{ \mtx Z_{i-1} }_p^2 \big)^{1/2}
	\leq \sigma_i m_i \cdot \triplenorm{ \mtx{Z}_{i-1} }_{p,2}.
\end{align*}
We have used the fact that $\mtx{Y}_i$ is independent from $\mtx{P}_{i-1}$ and from $\mtx{Z}_{i-1}$
to pass to the last line.

The rest of the proof runs along the same lines as the argument in Theorem~\ref{thm:pnorm},
using the last display in place of the bound~\eqref{eqn:fluctuation-bd}.
\end{proof}

Let us offer a simple example to illustrate why  Theorem~\ref{thm:compression_projection}
can produce better outcomes than Theorem~\ref{thm:pnorm}.
Consider a random matrix $\mtx{X} \in \M_d$ with the distribution
$\Prob{ \smash{\mtx{X} = \mathbf{e}_j \mathbf{e}_j{}^*} } = d^{-1}$ for each $j = 1, \dots ,d$.
As usual, $\mathbf{e}_j \in \CC^d$ is the $j$th standard basis vector.
Construct the random matrix $\mtx{Y} = \Id + \eps \mtx{X}$, where
$\eps$ is a Rademacher random variable that is independent from $\mtx{X}$.
Clearly, $\E \mtx{Y} = \Id$.  For any rank-$r$ orthogonal projector $\mtx{P}$,
\[
\E \norm{ (\mtx{Y} - \E \mtx{Y}) \mtx{P} }^2 = \E \norm{ \mtx{P} \mtx{X}^* \mtx{X} \mtx{P} }
	= \frac{1}{d} \sum_{i=1}^d \trace[ \mtx{P} \mathbf{e}_i \mathbf{e}_i^* \mtx{P} ]
	= \frac{1}{d} \trace \mtx{P}
	= \frac{r}{d}.
\]
Therefore,
\[
\sup_{\mtx{P} \in \mathcal{P}_r} \left( \E \norm{ (\mtx{Y} - \E \mtx{Y}) \mtx{P} }^2 \right)^{1/2}
	= \sqrt{r/d} \leq 1.
\]
By contrast, $\E \|\mtx Y - \E \mtx Y\|^2 = \E \|\mtx X\|^2 = 1$.
When $r \ll d$, this bound offers a significant improvement.
instead of the ambient dimension $d$.

\subsection{Adapted Sequences}
\label{sec:adapted}

We can easily generalize our results on a product of independent random matrices
to a product of adapted random matrices.  This kind of extension is valuable for studying
iterative algorithms where the choices made by the algorithm at a given step
depend on the history of the iteration.  

Let $(\Omega, \mF, \mathbbm{P})$ be a probability space,
and let $\mF_1 \subset \mF_2 \subset \cdots \subset \mF_n \subset \mF$ be a filtration.
For each index $i = 1, \dots, n$, we write $\E_i$ for the expectation conditioned
on the $\sigma$-algebra $\mF_i$.  The operator $\E_0 := \E$ is the unconditional expectation.

We consider an adapted sequence $\{ \mtx{Y}_1, \dots, \mtx{Y}_n \} \subset \M_d$ of random
matrices; that is, each $\mtx{Y}_i$ is measurable with respect to $\mF_i$.
The next result provides information about the growth and concentration
properties of the product $\mtx{Z}_n = \mtx{Y}_n \cdots \mtx{Y}_1$.
Note that the natural concentration result compares $\mtx{Z}_n$ with a product
of conditional expectations, rather than the expectation of the product.

\begin{theorem}[Products of Adapted Random Matrices] \label{thm:pnorm_adapted}
Consider a fixed matrix $\mtx{Z}_0 \in \M_d$ and an adapted sequence $\{ \mtx{Y}_1, \dots, \mtx{Y}_n \} \subset \M_d$
of random matrices.  Form the products
\[
\mtx{Z}_n = \mtx{Y}_n \cdots \mtx{Y}_1 \mtx{Z}_0
\quad\text{and}\quad
\mtx{F}_n = (\E_{n-1} \mtx{Y}_n) \cdots (\E_1 \mtx{Y}_2) (\E_{0} \mtx{Y}_1) \mtx{Z}_0.
\]
Assume that
\[
\norm{ \E_{i-1} \mtx{Y}_i } \leq m_i
\quad\text{and}\quad
\norm{ \mtx{Y}_i - \E_{i-1} \mtx{Y}_i } \leq \sigma_i m_i
\quad\text{almost surely for $i = 1, \dots, n$.}
\]
Define $M = \prod_{i=1}^n m_i$ and $v = \sum_{i=1}^n \sigma_i^2$.
For $2 \leq q\leq p$, the random product $\mtx{Z}_n$ satisfies the growth
and concentration bounds
\begin{align}
\triplenorm{ \mtx{Z}_n }_{p,q}
	&\leq \econst^{\cnst{C}_p v / 2} \norm{ \mtx{Z}_0 }_p \cdot M;
	\label{eqt:expectation_adapted_1} \\
\triplenorm{ \mtx{Z}_n - \mtx{F}_n }_{p,q}
	&\leq \big( \econst^{\cnst{C}_p v} - 1 \big)^{1/2} \norm{\mtx{Z}_0}_p \cdot M.
	\label{eqt:expectation_adapted_2}
\end{align}
\end{theorem}

\begin{proof}
Recursively construct the products
\[
\mtx{Z}_i = \mtx{Y}_i \mtx{Z}_{i-1}
\quad\text{and}\quad
\mtx{F}_i = (\E_{i-1} \mtx{Y}_i) \mtx{Z}_{i-1}
\quad\text{for $i = 1, \dots, n$.}
\]
To bound the growth of $\mtx{Z}_i$ and the concentration of $\mtx{Z}_i - \mtx{F}_i$,
we simply need to update the argument from Theorem~\ref{thm:pnorm}.

To obtain~\eqref{eqt:expectation_adapted_1}, decompose
\[
\mtx{Z}_i = (\E_{i-1} \mtx{Y}_i)\mtx{Z}_{i-1} + (\mtx{Y}_i - \E_{i-1} \mtx{Y}_i) \mtx{Z}_{i-1}.
\]
Since $\E_{i-1} \mtx{Y}_i$ and $\mtx{Z}_{i-1}$ are both measurable with respect to $\mF_{i-1}$
and $\E_{i-1}(\mtx{Y}_i - \E_{i-1} \mtx{Y}_i) = \mtx{0}$, the obvious variant of
Proposition~\ref{prop:Ricard_Xu} implies that
\[
\begin{aligned}
\triplenorm{ \mtx{Z}_i }_{p,q}^2
	&\leq \triplenorm{ (\E_{i-1} \mtx{Y}_i) \mtx{Z}_{i-1} }_{p,q}^2 + \cnst{C}_p \triplenorm{ (\mtx{Y}_i - \E_{i-1} \mtx{Y}_i) \mtx{Z}_{i-1} }_{p,q}^2 \\
	&\leq m_i^2 \triplenorm{ \mtx{Z}_{i-1} }_{p,q} + \cnst{C}_p m_i^2 \sigma_i^2 \triplenorm{ \mtx{Z}_{i-1} }_{p,q}^2.
\end{aligned}
\]
The second inequality follows from~\eqref{eqn:operator-ideal}.
This is the same recurrence we obtain in the proof of Theorem~\ref{thm:pnorm}, relation~\eqref{eq:master_norm1}.
The rest of the argument is the same.

To obtain~\eqref{eqt:expectation_adapted_2}, decompose
\[
\mtx{Z}_i - \mtx{F}_i = \mtx{Y}_i \mtx{Z}_{i-1} - (\E_{i-1} \mtx{Y}_i) \mtx{F}_{i-1}
	= (\E_{i-1} \mtx{Y}_i)(\mtx{Z}_{i-1} - \mtx{F}_{i-1}) + (\mtx{Y}_i - \E_{i-1}\mtx{Y}_i) \mtx{Z}_{i-1}.
\]
As before, Proposition~\ref{prop:Ricard_Xu} implies that
\[
\begin{aligned}
\triplenorm{ \mtx{Z}_i - \mtx{F}_i }_{p,q}^2
	&\leq \triplenorm{ (\E_{i-1} \mtx{Y}_i) (\mtx{Z}_{i-1} - \mtx{F}_{i-1}) }_{p,q}^2 + \cnst{C}_p \triplenorm{ (\mtx{Y}_i - \E_{i-1} \mtx{Y}_i) \mtx{Z}_{i-1} }_{p,q}^2 \\
	&\leq m_i^2 \triplenorm{ \mtx{Z}_{i-1} - \mtx{F}_{i-1} }_{p,q} + \cnst{C}_p m_i^2 \sigma_i^2 \triplenorm{ \mtx{Z}_{i-1} }_{p,q}^2.
\end{aligned}
\]
This is the same recurrence that arose when we established Theorem~\ref{thm:pnorm}, relation \eqref{eq:master_norm2}.
The balance of the argument is identical.
\end{proof}

\subsection{The Spectral Radius}
\label{sec:spectral-radius}

Products of matrices are closely related to the evolution of discrete-time linear
dynamical systems.  In this context, it may be more natural to study the
\emph{spectral radius} of the matrix product, rather than its spectral norm.
Bounds for the spectral radius follow as corollary of our work,
owing to the following classical fact.

\begin{fact}[Schur] \label{fact:schur}
Let $\mtx{M} \in \M_d$ be a square matrix.  The spectral radius $\varrho(\mtx{M})$
is defined as the maximum absolute value of an eigenvalue of $\mtx{M}$.
It satisfies the variational principle
$$
\varrho(\mtx{M}) = \inf_{\mtx{S} \in \M_d} \norm{ \mtx{S}^{-1} \mtx{M} \mtx{S} }.
$$
The infimum takes place over all invertible matrices $\mtx{S}$.  In particular $\varrho(\mtx{M}) \leq \norm{\mtx{M}}$.
\end{fact}

Let us give an indication of the kinds of results that are possible.

\begin{corollary}[Expectation Bounds for the Spectral Radius] \label{cor:expectation-srad}
Consider an independent sequence $\{ \mtx{Y}_1, \dots, \mtx{Y}_n \} \subset \M_d$
of random matrices, and form the product $\mtx{Z}_n = \mtx{Y}_n \cdots \mtx{Y}_1$.
Let $\mtx{S} \in \M_d$ be a fixed invertible matrix, and assume that
\[
\norm{ \mtx{S}^{-1} (\Expect \mtx{Y}_i) \mtx{S} } \leq m_i
\quad\text{and}\quad
\Big( \Expect \norm{ \mtx{S}^{-1}( \mtx{Y}_i - \Expect \mtx{Y}_i) \mtx{S} }^2 \Big)^{1/2} \leq \sigma_i m_i
\quad\text{for $i = 1, \dots, n$.}
\]
Let $M = \prod_{i=1}^n m_i$ and $v = \sum_{i=1}^n \sigma_i^2$.  Then
\begin{align*}
\E \varrho( \mtx{Z}_n )
	&\leq \exp\left( \sqrt{2 v \, (2 v \vee \log d)} \right) \cdot M. 
\end{align*}
\end{corollary}

\begin{proof}
Combine Corollary~\ref{cor:expectation} and Fact~\ref{fact:schur}.
\end{proof}

\subsection{Prospects}

We have developed a collection of nonasymptotic bounds for products
of random matrices.  These results hold under simple and easily verifiable
conditions, and they give accurate predictions about the behavior of
some particular instances (e.g., products of iid random perturbations of the identity).
The proofs are based on foundational results about the geometry
of the Schatten classes, and they can easily be adapted to treat variants
of the problems under consideration.

A disappointing feature of our results is that they do not account for interactions between the matrix factors.
For example, when $\mtx Y_i = \Id + \mtx X_i/n$ for bounded, independent matrix perturbations $\mtx X_i$, we have shown that
\begin{equation*}
\log \E \|\mtx Y_n \cdots \mtx Y_1\| \leq \frac 1n \sum_{i=1}^n \| \E \mtx X_i\| + O\left(\sqrt{\frac{\log d}{n}}\right)\,.
\end{equation*}
However, when the matrices $\mtx X_i$ commute almost surely, it is easy to show the sharper bound
\begin{equation*}
\log \E \|\mtx Y_n \cdots \mtx Y_1\| \leq \frac 1n \left\|\sum_{i=1}^n  \E \mtx X_i\right\| + O\left(\sqrt{\frac{\log d}{n}}\right)\,.
\end{equation*}
The results of Emme and Hubert~\cite{emme2017limit} establish that $\lim_{n \to \infty} \log \E \|\mtx Y_n \cdots \mtx Y_1\| = \lim_{n \to \infty} \left\|\sum_{i=1}^n  \E \mtx X_i\right\|/n$.
It therefore seems reasonable to conjecture that a refined bound of the latter type exists in more generality.
The growth bounds discussed in Remark 5.2 imply a statement of the form
\begin{equation*}
\log \E \|\mtx Y_n \cdots \mtx Y_1\| \leq \log  \frac 1n \left\|\prod_{i=1}^n \E \mtx X_i \right\| + \mathrm{error}\,,
\end{equation*}
but the error term is not sharp.
This type of bound would echo Tropp's improvements~\cite{Tro12}
to the Ahlswede--Winter results~\cite{AW02:Strong-Converse} for a sum
of independent random matrices.  At present, it is not clear whether
this refinement is possible, nor what technical arguments would lead there.
\appendix

\section{Supplementary Proofs}

This appendix collects a few additional arguments.
First, we establish the sharp form of the result
on subquadratic averages, Proposition~\ref{prop:Ricard_Xu},
using an elementary method.

\begin{lemma}[Sharp Subquadratic Averages]
\label{lem:subquadratic-sharp}
Let $\mtx{X}, \mtx{Y}$ be random matrices of the same size that satisfy $\Expect[ \mtx{Y} | \mtx{X} ] = \mtx{0}$.
When $2 \leq q \leq p$,
\[
\triplenorm{ \mtx{X} + \mtx{Y} }_{p,q}^2 \leq \triplenorm{ \mtx{X} }_{p,q}^2 + \cnst{C}_p \triplenorm{ \mtx{Y} }_{p,q}^2,
\]
where the optimal constant $\cnst{C}_p := p - 1$.
\end{lemma}

\begin{proof}
Fix a natural number $n$, and set $\mtx{Z} = n^{-1} \mtx{Y}$.  Inequality~\eqref{eqn:subquadratic-weak} states that
\[
D_1 := \triplenorm{ \mtx{X} + \mtx{Z} }_{p,q}^2 - \triplenorm{ \mtx{X} }_{p,q}^2 - 2 \cnst{C}_p \triplenorm{\mtx{Z}}_{p,q}^2 \leq 0.
\]
For a parameter $2 \leq k \leq n$, Corollary~\ref{cor:2smooth} and Lyapunov's inequality imply that
\[
\triplenorm{ \mtx{X} + k \mtx{Z} }_{p,q}^2 + \triplenorm{ \mtx{X} + (k-2) \mtx{Z} }_{p,q}^2
	\leq 2 \triplenorm{ \mtx{X} + (k-1) \mtx{Z} }_{p,q}^2 + 2 \cnst{C}_p \triplenorm{ \mtx{Z} }_{p,q}^2.
\]
Rearranging the last display, we see that
\[
\begin{aligned}
D_k &:= \triplenorm{ \mtx{X} + k \mtx{Z} }_{p,q}^2 - \triplenorm{ \mtx{X} + (k-1)\mtx{Z} }_{p,q}^2 - 2 \cnst{C}_p k \triplenorm{\mtx{Z}}_{p,q}^2 \\
	&\phantom{:}\leq \triplenorm{ \mtx{X} + (k-1) \mtx{Z} }_{p,q}^2 - \triplenorm{ \mtx{X} + (k-2)\mtx{Z} }_{p,q}^2 - 2 \cnst{C}_p (k-1) \triplenorm{\mtx{Z}}_{p,q}^2
	= D_{k-1}.
\end{aligned}
\]
In particular, $D_k \leq D_1 \leq 0$.  Using a telescoping sum,
\begin{align*}
\triplenorm{ \mtx{X} + \mtx{Y} }_{p,q}^2 - \triplenorm{ \mtx{X} }_{p,q}^2
	&= \sum_{k=1}^n \left( \triplenorm{ \mtx{X} + k \mtx{Z} }_{p,q}^2 - \triplenorm{ \mtx{X} + (k-1) \mtx{Z} }_{p,q}^2 \right) \\
	&= \sum_{k=1}^n \left( D_k + 2\cnst{C}_p k \triplenorm{\mtx{Z}}_{p,q}^2 \right)
	\leq \sum_{k=1}^n 2 \cnst{C}_p k \triplenorm{\mtx{Z}}_{p,q}^2
	= \cnst{C}_p \frac{n+1}{n} \triplenorm{ \mtx{Y} }_{p,q}^2.
\end{align*}
Take the limit as $n \to \infty$ to arrive at the stated result.
\end{proof}

Second, we present a basic numerical inequality for weighted
sums of exponentials.

\begin{lemma}\label{lem:number_inequality}
Let $a_1,a_2,\dots,a_n$ be a sequence of real numbers. Then
\begin{equation}
\sum_{i=1}^na_i \exp\left( \sum_{k=1}^{i-1}a_k \right)\leq \exp\left( \sum_{i=1}^na_i \right) - 1.
\end{equation}
\end{lemma}

\begin{proof}
The elementary inequality $a \leq \econst^a - 1$, valid for $a \in \R$, implies that
\begin{equation*}
a_i \exp\left( \sum_{k=1}^{i-1} a_k \right) \leq \exp\left( \sum_{k=1}^{i} a_k \right)  - \exp\left( \sum_{k=1}^{i-1} a_k \right).
\end{equation*}
Sum the displayed equation over $i = 1, \dots, n$ to verify the claim.
\end{proof}

\bibliographystyle{habbrv}
\bibliography{products_new.bib}

\end{document}